\documentclass[12pt]{amsart} 

\usepackage[all,cmtip,2cell]{xy}

\usepackage[inline]{enumitem}
\usepackage{amsmath,amsthm,amssymb}
\usepackage{words,notn}
\usepackage{geometry}
\usepackage[colorlinks=true]{hyperref}

\usepackage{comment}
\usepackage{notn}

\def\notndescdimen{.65\textwidth}

\providecommand{\Kim}{\cat{Kim}}
\providecommand{\Kimstab}{\Kim^{\rm stab}}
\providecommand{\JLi}{\cat{Li}}
\providecommand{\JListab}{\JLi^{\rm stab}}
\providecommand{\KimACGS}{\Kim\ACGS}
\providecommand{\KimACGSstab}{\KimACGS^{\rm stab}}

\providecommand{\ACGS}{\cat{ACGS}}
\providecommand{\ACGSstab}{\ACGS^{\rm stab}}
\providecommand{\Log}{\cat{Log}}
\providecommand{\AF}{\cat{AF}}
\providecommand{\AFstab}{\cat{AF}^{\rm stab}}
\providecommand{\Cad}{\cat{Cad}}

\providecommand{\Kstab}{K^{\rm stab}}
\providecommand{\Lstab}{L^{\rm stab}}

\providecommand{\fMB}{\mathfrak{MB}}
\providecommand{\Mnd}{\fM^{\rm{nondeg}}}

\newcommand{\marg}[1]{\normalsize{{\footnote{{#1}}}}{\marginpar[\hfill\tiny\thefootnote$\rightarrow$]{{$\leftarrow$\tiny\thefootnote}}}}
\renewcommand{\marg}[1]{}

\newcommand{\steffen}[1]{\marg{{Steffen:} #1}}
\newcommand{\Jonathan}[1]{\marg{{Jonathan:} #1}}

\newcommand{\Dan}[1]{\marg{{Dan:} #1}}

\newtheoremstyle{named}%
	{}%
	{}%
	{\itshape}%
	{}%
	{\bfseries}%
	{.}%
	{.5em}%
	{\thmname{#1 #3}}
\newcommand{\namer}[3]{{#1} {#3} for {#2}}
\newcommand{\namerr}[3]{\namer{\ref*{#1}}{#2}{\ref*{#3}}}

\newtheorem{theorem}{Theorem}
\newtheorem{proposition}[theorem]{Proposition}
\newtheorem{corollary}[theorem]{Corollary}
\newtheorem{lemma}[theorem]{Lemma}
\newtheorem{mainlemma}{Lemma}
\theoremstyle{remark}
\newtheorem{remark}[theorem]{Remark}
\theoremstyle{named}
\newtheorem{nlem}{Lemma}
\theoremstyle{definition}
\newtheorem{definition}[theorem]{Definition}

\numberwithin{theorem}{subsection}

\begin{document}

\title[Comparison theorems for Gromov--Witten invariants]{Comparison theorems for Gromov--Witten invariants of smooth pairs and of degenerations}

\author{Dan Abramovich}

\author{Steffen Marcus}

\author{Jonathan Wise}

\address[Abramovich]{Department of Mathematics\\
Brown University\\
Box 1917\\
Providence, RI 02912\\
U.S.A.}
\email{abrmovic@math.brown.edu}

\address[Marcus]{Department of Mathematics\\
University of Utah\\
155 S 1400 E RM 233\\
Salt Lake City, UT 84112-0090\\
U.S.A.}

\email{marcus@math.utah.edu}

\address[Wise]{Department of Mathematics\\
University of Colorado at Boulder\\
Campus Box 395\\
Boulder, CO 80309-0395\\
U.S.A.}

\email{jonathan.wise@colorado.edu}

\thanks{Abramovich supported in part by NSF grant
  DMS-0901278. Marcus supported in part by funds from NSF grant
  DMS-0901278. Wise supported by an NSF postdoctoral fellowship.}
%\date{\today}
\begin{abstract} We consider four approaches to relative Gromov--Witten theory and Gromov--Witten theory of degenerations:  J. Li's original approach, B.\ Kim's logarithmic expansions, Abramovich--Fantechi's orbifold expansions, and a logarithmic theory without expansions due to Gross--Siebert and Abramovich--Chen.  We exhibit morphisms relating these moduli spaces and prove that their virtual fundamental classes are compatible by pushforward through these morphisms.  This implies that the Gromov--Witten invariants associated to all four of these theories are identical.
\end{abstract} 
\maketitle

\setcounter{tocdepth}{1}
\setcounter{secnumdepth}{4}
\tableofcontents

\section{Introduction}\label{sec:intro}

The extraordinary feature of Gromov--Witten invariants distinguishing them from the enumerative invariants with which they sometimes fraternize is their deformation invariance.  However, this feature is not as powerful as one might hope, as Gromov--Witten invariants were initially defined only for smooth targets.  Indeed, one would like to degenerate a complicated space to one that is singular but geometrically simpler, replacing geometric complexity by the combinatorics of the irreducible components.  

The very existence of a definition that is invariant under deformation suggests that it might be possible to extend that definition of Gromov--Witten invariants to one that applies also to singular targets, and indeed, many have sought such generalizations.  For the mildest of singularities---two smooth schemes meeting along a smooth divisor---there have been a number of successes.  Gromov--Witten invariants for such singularities were defined in the symplectic setting by A.\ M.\ Li--Ruan \cite{Li-Ruan} and Ionel--Parker \cite{Ionel-Parker, Ionel-Parker2}, and algebraically by J.\ Li \cite{Li1,Li2}.  These authors also proved a \emph{degeneration formula} that showed how these invariants could be recovered from \emph{relative Gromov--Witten invariants} depending on each of the two smooth schemes and the divisor along which they meet.

The impact of the degeneration formula has been profound, but efforts to generalize the theory to other degenerations have until recently met with limited success (see \cite{Parker,Parker1,Chen,GS,AC} for recent progress). One of the obstacles has been the difficulty of Li's theory, which has two rather subtle aspects:
\begin{enumerate*}[label=(\roman{*})]
\item to avoid badly behaved deformation spaces, Li expands the target variety and the choice of expansion is allowed to vary in families; and
\item the deformation theory of maps from curves into Li's expansions presents a number of complications.
\end{enumerate*}

With an eye towards generalizations, there have been several attempts to simplify Li's methods.  It has been clear for some time, beginning with an early lecture of Siebert \cite{Siebert}, that logarithmic geometry should play a role in a general theory of stable maps to degenerations and mildy singular targets.  However, Siebert's proposal awaited the development of logarithmic algebraic geometry, particularly the logarithmic cotangent complex, on which it relied to define virtual fundamental classes.\footnote{With the benefit of hindsight, we now know that, by working relative to a universal target, it is possible to give another construction of the virtual fundamental class that avoids the logarithmic cotangent complex.  The equivalence of this construction with the one based on the cotangent complex is one of our main tools in the proof of Theorem~\ref{thm:compare}.} \Dan{slight rephrase} \Jonathan{another slight rephrase}

In the interim, Cadman saw that orbifold Gromov--Witten theory~\cite{CR,AV,AGV} could be used for some of the same purposes as Li's relative Gromov--Witten theory~\cite{cadman}.  Abramovich and Fantechi adapted and generalized Cadman's method to apply to Li's expanded targets.  At the cost of reintroducing Li's expansions, they were able to develop a theory in which predeformability could be replaced by a more innocuous transversality condition~\cite{AF}.

After the logarithmic cotangent complex became available in the work of Olsson~\cite{olsson-log-cc}, Kim used it to study logarithmic maps from curves into expanded targets~\cite{Kim}, giving another means to avoid predeformability.  Li's expansions remained, though, in both the orbifold theory of Abramovich--Fantechi and the logarithmic theory of Kim.  They were not removed until Siebert's program was realized in the works of Chen, Gross--Siebert, and Abramovich~\cite{Chen,GS,AC}.  

Now we have five virtual counting theories for curves relative to a divisor and four virtual counting theories for curves on mildly singular targets (Cadman's approach does not apply at present to singular targets).  Our purpose here is to compare the four defined by Li, Abramovich--Fantechi, Kim, and Abramovich--Chen and Gross--Siebert, which we denote provisionally by $\JLi$, $\AF$, $\Kim$, and $\ACGS$, respectively (more specific notation will be introduced in the sequel).

\numberwithin{theorem}{section}
\begin{theorem} \label{thm:compare}
There are maps (see section~\ref{sec:maps}),
\[ \label{eqn:compare}
 \xymatrix{
     \AF  \ar[dr]^{\Psi} & & \Kim \ar[dl]_{\Theta} \ar[dr]^{\Upsilon} & \\
  & \JLi &  & \ACGS .
}
\]
such that, for each of the maps $p : K \rightarrow L$ above, $p_\ast [K]^{\vir} = [L]^{\vir}$.  That is,
\begin{align*}
\Psi_\ast [\AF]^{\vir} &= [\JLi]^{\vir} &
\Theta_\ast [\Kim]^\vir &= [\JLi]^\vir &
\Upsilon_\ast [\Kim]^\vir &= [\ACGS]^\vir .
\end{align*}
\end{theorem}

\begin{corollary}
The primary and descendant Gromov--Witten invariants associated to $\AF$, $\JLi$, $\Kim$, and $\ACGS$ are all identical.
\end{corollary}
\numberwithin{theorem}{subsection}

\noindent More precise statements of these results appear in Theorems \ref{thm:compare-pairs} and \ref{thm:compare-deg} and their corollaries below.

A comparison involving Cadman's approach is more elusive.  It is shown in \cite{ACW} that there is a morphism $\Phi:\AF \to \Cad$ such that $\Phi_*[\AF]^{\vir} = [\Cad]^{\vir}$, provided that the genus is restricted to $g=0$ and the orbifold structure is sufficiently twisted.  Counterexamples abound in the absence such conditions.

We hope that, apart from offering a demonstration of Theorem~\ref{thm:compare}, this paper will also serve as an illustration of two techniques used in the proof:  We have made systematic use throughout this paper of Costello's comparison theorem~\cite[Theorem~5.0.1]{costello} and the obstruction theory formalism of~\cite{obs}.
%Apart from offering a proof of Theorem~\ref{thm:compare}, we hope this paper will also serve as an illustration of two techniques:  we make systematic use of Costello's comparison theorem~\cite[Theorem~5.0.1]{costello} and the obstruction theory formalism of~\cite{obs}. 

We work over the complex numbers $\bC$.

\subsection{Gromov--Witten theories for smooth pairs}\label{sec:sm-pair}

\begin{definition} \label{def:sm-pair}
A \emph{smooth pair} is a pair $(X,D)$ such that $X$ is a smooth scheme and $D$ is a smooth divisor in $X$.
\end{definition}

The discrete invariants for the various moduli spaces above are described in \cite{Li1} (also, see \cite[Convention~3.1.1-3.1.2]{AF}) as follows:
\begin{itemize}
\item the genus $g\in\bZ_{\geq 0}$;
\item the curve class $\beta\in H_2(X,\bZ)$;
\item the number of marked points $m,n \in\bZ_{\geq 0}$\Jonathan{switched order of $m$ and $n$} (counting marked points inside of and outside of $D$, respectively); and
\item an integer partition $\alpha_1+\cdots+\alpha_m \:\vdash \:\beta. D$ prescribing contact order along $D$.
\end{itemize}
We bundle these invariants using the notation $\Gamma = (g,n,\beta,\alpha_1\ldots,\alpha_m)$. 

For each $K \in \set{\AF,\JLi,\Kim,\ACGS}$ and each choice of discrete data $\Gamma$ as above, we have a proper Deligne--Mumford stack $\Kstab_\Gamma(X,D)$ compactifying the space of maps from smooth curves of genus $g$ to $X$ with specified homology class $\beta$ and specified orders of contact along $D$.
\Jonathan{made more specific}

\begin{theorem}\label{thm:compare-pairs}
For any fixed choice of discrete invariants $\Gamma$, we have:
\begin{align*}
{\Psi}_\ast\left[\AFstab_\Gamma(X,D)\right]^{\rm vir}&=\left[\JListab_\Gamma(X,D)\right]^{\rm vir}\\
{\Theta}_\ast\left[\Kimstab_\Gamma(X,D)\right]^{\rm vir}&=\left[\JListab_\Gamma(X,D)\right]^{\rm vir}\\
{\Upsilon}_\ast\left[\Kimstab_\Gamma(X,D)\right]^{\rm vir}&=\left[\ACGSstab_\Gamma(X,D)\right]^{\rm vir}.\\
\end{align*}
\end{theorem}

The following corollary is immediate.
\begin{corollary} \label{cor:GW1}
All primary relative Gromov--Witten invariants defined by the moduli spaces $\AFstab(X,D)$, $\JListab(X,D)$, $\Kimstab(X,D)$, and $\ACGSstab(X,D)$ coincide.  Descendant invariants also coincide if we use $\psi$-classes coming from the underlying \emph{unexpanded} stable maps to $X$ (see section \ref{sec:comp-GW} for the definitions).
\end{corollary}

\subsection{Gromov--Witten theories for acceptable degenerations}\label{sec:acc-deg}

At least morally, an acceptable degeneration is a family whose general fiber and total space are smooth and whose special fiber is the union of two smooth schemes along a smooth divisor.  However, it is preferable to have a definition that permits arbitrary changes of base, so we use the following:
\begin{definition} \label{def:acc-deg}
An \emph{acceptable degeneration} is a flat family $p : X \rightarrow V$ equipped with 
\begin{enumerate}[label=(\roman{*})]
\item line bundles and sections $(L_1, s_1)$ and $(L_2, s_2)$ on $X$, corresponding to a map $X \rightarrow \sA^2$,
\item a line bundle and section $(M, t)$ on $V$, corresponding to a map $V \rightarrow \sA$, and
\item an isomorphism $p^\ast (M,t) \simeq (L_1 \tensor L_2, s_1 \tensor s_2)$,
\end{enumerate}
subject to the condition that
\begin{enumerate}[label=(\roman{*}), resume]
\item the induced map $X \rightarrow \sA^2 \fp_{\sA} V$ be smooth.
\end{enumerate}
\end{definition}
\noindent In the definition, $\sA$ is the stack $[\,\bA^1 \,/\, \Gm\,]$.  See \cite[Section~2.2]{ACFW} for more about this definition.

The discrete data are slightly simpler here than they were for pairs.  They consist of
\begin{itemize}
\item the genus $g\in\bZ_{\geq 0}$;
\item the curve class $\beta\in H_2(X,\bZ)$; and
\item the number of marked points $n \in\bZ_{\geq 0}$.
\end{itemize}
We write $\Gamma=(g,n,\beta)$.

For each $K \in \set{\AF,\JLi,\Kim,\ACGS}$ let $\Kstab_\Gamma(X/V)$ 
\Jonathan{added $\Gamma$}
be the corresponding family of stable maps to the family $X/V$.

\begin{theorem}\label{thm:compare-deg}
For any fixed choice of discrete invariants $\Gamma$,
\begin{align*}
{\Psi}_\ast\left[\AFstab_\Gamma(X/V)\right]^{\rm vir}&=\left[\JListab_\Gamma(X/V)\right]^{\rm vir}\\
{\Theta}_\ast\left[\Kimstab_\Gamma(X/V)\right]^{\rm vir}&=\left[\JListab_\Gamma(X/V)\right]^{\rm vir}\\
{\Upsilon}_\ast\left[\Kimstab_\Gamma(X/V)\right]^{\rm vir}&=\left[\ACGSstab_\Gamma(X/V)\right]^{\rm vir}.\\
\end{align*}
\end{theorem}

Once again we have an immediate corollary:
\begin{corollary} \label{cor:GW2}
All primary relative Gromov--Witten invariants defined by the moduli spaces $\AFstab(X/V)$, $\JListab(X/V)$, $\Kimstab(X/V)$, and $\ACGSstab(X/V)$ coincide.  Descendant invariants also coincide if we use $\psi$-classes coming from the underlying \emph{unexpanded} stable maps to $X$(again see section \ref{sec:comp-GW}).
\end{corollary}

\subsection{Obstruction theories}  \label{sec:obs}

Throughout this paper we use the obstruction theory formalism introduced in~\cite{obs}.  We summarize the definition and the properties we will use below.  This definition is shown in loc.\ cit.\ to be equivalent to the one given by Behrend and Fantechi~\cite{BF}, but the formulation we use here avoids reference to the cotangent complex and thereby simplifies many of the verifications we will need to make later.

Suppose that $X \rightarrow Y$ is a morphism of algebraic stacks.  A relative obstruction theory for $X$ over $Y$ consists of the following data:
\begin{enumerate}[label=(\roman{*})]
\item for each $X$-scheme $S$ and each quasi-coherent sheaf $J$ on $S$, an abelian $2$-group $\sE(S, J)$,
\item for each square-zero extension $S'$ of $S$ over $Y$ with ideal $I_{S/S'} = J$, an element, known as the {\em obstruction}, $\omega \in \sE(S,J)$, and
\item an identification between the sheaf $\uIsom_{\sE(S,J)}(0, \omega)$ and the sheaf of lifts of the diagram
\begin{equation} \label{eqn:13} \vcenter{\xymatrix{
S \ar[r] \ar[d] & X \ar[d] \\
S' \ar@{-->}[ur] \ar[r] & Y .
}} \end{equation}
\end{enumerate}
These data are required to satisfy a number of compatibility conditions that we suppress here.  See \cite{obs} for the complete definition.  Typically these conditions are easy to verify.\footnote{An exception to this rule is J.\ Li's obsruction theory for stable maps to expanded degenerations and expanded pairs.  See~\cite{Li2} or~\cite[Appendix]{CMW} for discussion of this obstruction theory and the verification of the relevant properties.}

\subsubsection{Constructing obstructions} \label{sec:con-obs}

All of the obstruction theories employed in this paper can be viewed as examples of the following abstract construction.  Suppose that there is a site $\fX$ and a morphism of sites $\fX \rightarrow \et(X)$ such that 
\begin{enumerate*}[label=(\roman{*})]
\item the lifting problem~\eqref{eqn:13} can be posed for $S \in \fX$,\footnote{There is not necessarily a unique way of making sense of the deformation problem in $\fX$, though there is often a natural one.}
\item the lifting problem always admits a solution locally in $\fX$, and
\item any two solutions to the lifting problem are locally isomorphic.
\end{enumerate*}
An example may help to illuminate these hypotheses.  If $X$ is the moduli space of stable maps to a smooth target $Z$ (and $Y$ is a point) we take $\fX$ to be the \'etale site of the universal curve (see \cite[Section~7.3]{obs}).  For J.\ Li's moduli space of relative stable maps $\fX$ will be a hybrid of the \'etale sites of the universal curve and the base.

In all reasonable situations (including moduli spaces parameterizing morphisms of schemes or flat families of schemes), it follows formally that when $S' = S[J]$ is the trivial square-zero extension, the stack of lifts of~\eqref{eqn:13} in $\fX$ is a stack of commutative $2$-groups\footnote{Commutative $2$-groups have gone by various names elsewhere, including ``strictly commutative Picard categories'' \cite[XVIII.1.4]{sga4-3}.} on $\fX$.  Denoting this stack of $2$-groups by $T$, it is equally formal to show that the lifts of~\eqref{eqn:13} for \emph{any} square-zero extension $S'$ of $S$ by $J$ form a torsor on $\fX$ under $T$.\footnote{See \cite[Section~6]{Breen} for the notion of a torsor under a stack of $2$-groups.  Note that $2$-groups are called gr-stacks in loc.\ cit.}  Therefore if $\sE(S,J)$ is defined to be the $2$-category of all torsors on $\fX$ under $T$, there is an obstruction $\omega \in \sE(S,J)$---the torsor itself---to the existence of a lift of~\eqref{eqn:13}.  That is $\sE$ forms a relative obstruction theory for $X$ over $Y$.

In general, the torsors under a $2$-group $G$ form a $2$-category, which possesses the structure of a commutative $3$-group if $G$ is commutative.  In this paper we will usually have a stability condition that ensures the $2$-category is actually a category and that the $3$-group is therefore a $2$-group.  However, at one point (section~\ref{sec:AF-vfc}) we will encounter an obstruction theory for a morphism of Artin type.  Fortunately, this obstruction theory turns out to be rather trivial and we can deal with it without using $2$-categories in an explicit way (see section~\ref{sec:loc-unobs} for more about this).

\subsubsection{Perfect obstruction theories}

An obstruction theory $\sE$ for $X$ over $Y$ will be called \emph{perfect} if on a smooth cover $\tX$ of $X$ it is possible to find a perfect complex $\bE^\bullet$ on $\tX$ in cohomological degrees $[-1,0]$ and a functorial equivalence
\begin{equation*}
\sE(S,J) \simeq \Ext(f^\ast \bE^\bullet, J),
\end{equation*}
for all $f : S \rightarrow \tX$ and quasi-coherent $J$ on $S$.  Here $\Ext$ denotes the category of extensions of complexes:  see, for example, \cite[Section~4.1]{obs}, where the notation $\Psi_{\bE^\bullet}$ is used for $\Ext(\bE^\bullet, -)$.

By the construction of Behrend and Fantechi~\cite{BF}, if $X$ is of Deligne--Mumford type over $Y$ then a perfect relative obstruction theory gives rise to a relative virtual fundamental class, which we denote $[X/Y]^{\vir}$.  Following Manolache~\cite{Manolache}, a virtual fundamental class for $X$ over $Y$ can be viewed as a recipe for pulling back cycles from $Y$ to $X$.  If $g$ denotes the map $X \rightarrow Y$ then the virtual pullback is usually denoted $g^!_{\sE}$ or sometimes $g^!$ if it is safe to leave the dependence on $\sE$ tacit.
% (with the dependence on $\sE$ tacit).

\subsubsection{Pullback of obstruction theories}

Suppose that $X \rightarrow Y$ is a morphism of algebraic stacks and $\sE$ is a relative obstruction theory for $X$ over $Y$.  Let $Y' \rightarrow Y$ be a morphism of algebraic stacks, and set $X' = Y' \fp_Y X$.  Denote by $p$ the projection $X' \rightarrow X$.  For any $S \rightarrow X'$, define $\sE'(S,J) = \sE(S,J)$ where we view $S$ as an $X$-scheme by composition with $p$.  Then $\sE'$ is a relative obstruction theory for $X'$ over $Y'$.

\subsubsection{Compatible obstruction theories}

Suppose that $X \xrightarrow{g} Y \xrightarrow{f} Z$ is a sequence of morphisms of algebraic stacks.  Assume that there are relative obstruction theories $\sE'$ for $X/Y$, $\sE$ for $X/Z$, and $\sE''$ for $Y/Z$.  A \emph{compatibility} datum among these obstruction theories is written as an exact sequence
\begin{equation*}
0 \rightarrow \sE' \rightarrow \sE \rightarrow \sE'' \rightarrow 0
\end{equation*}
and consists of 
\begin{enumerate}[label=(\roman{*})]
\item functorial morphisms $\sE'(S,J) \xrightarrow{\alpha} \sE(S,J) \xrightarrow{\beta} \sE''(S,J)$ for each $S \rightarrow X$ and quasi-coherent $J$ on $S$, and
\item an identification of $\beta \circ \alpha$ with the zero map.
\end{enumerate}
These are required to satisfy the following conditions:
%\Dan{Does this really cover all the requirement of exactness??}\Jonathan{it looks fine to me; is something missing?}
\begin{enumerate}[label=(\roman{*}), resume]
\item $\alpha$ and $\beta$ are compatible with the natural maps between deformation problems,
\item for each $S$ and $J$, the map $\beta : \sE(S,J) \rightarrow \sE''(S,J)$ is locally surjective in $S$,
\item $\alpha$ induces an equivalence $\sE'(S,J) \simeq \ker(\beta)$.
\end{enumerate}
These conditions are spelled out in detail in \cite[Section~6.2]{obs}.

By \cite[Theorem~4.8]{Manolache}, if $f$ and $g$ are both of Deligne--Mumford type (so that the relative virtual fundamental classes can be defined) then $g^!_{\sE'} [Y/Z]^{\vir} = [X/Z]^{\vir}$.  By \cite[Proposition~6.5]{obs}, the same conclusion holds if $Y$ is locally unobstructed over $Z$ (see below).

\subsubsection{Locally unobstructed stacks}
\label{sec:loc-unobs}

At one occasion in this paper (section~\ref{sec:AF-vfc}), we will be obliged to study an obstruction theory for a morphism of Artin type.  Such obstruction theories are naturally families of commutative $3$-groups.  Fortunately, the morphism we shall study is \emph{lcoally unobstructed}, affording the reader who is uncomfortable with $3$-groups another point of view.

A morphism of algebraic stacks $Y \rightarrow Z$ with relative obstruction theory $\sE$ is called \emph{locally unobstructed} if for each $(S,J)$ the stack $\sE(S,J)$ is a gerbe over $S$ in the \'etale topology.  This means that every lifting problem~\eqref{eqn:13} admits a solution \'etale-locally in $S$.  In particular, $Y$ must be smooth over $Z$.  As the lifts of diagram~\eqref{eqn:13} are automatically a torsor under the relative tangent bundle stack $T_{Y/Z} \tensor J$ (see Appendix~\ref{app:tangent} for more about the tangent bundles of stacks), it follows that $\sE(S,J)$ is a gerbe over $S$, banded by $T_{Y/Z} \tensor J$.  Since $\sE(S,J)$ comes equipped with a section---the zero section---it is isomorphic to the stack of $(T_{Y/Z} \tensor J)$-torsors.  To say that the obstruction theory $\sE$ makes $Y$ unobstructed over $Z$ is therefore the same as to say that $Y$ is smooth over $Z$ and $\sE$ is the canonical obstruction theory.

Suppose that $X \rightarrow Y \rightarrow Z$ is a sequence of morphisms of algebraic stacks and 
\begin{equation*}
0 \rightarrow \sE' \rightarrow \sE \rightarrow \sE'' \rightarrow 0
\end{equation*}
is a sequence of compatible obstruction theories.  If $Y$ is locally unobstructed over $Z$, the maps $\sE'(S,J) \rightarrow \sE(S,J)$ are \'etale-locally surjective in $S$.  We may therefore rotate the sequence:  let $\sT''(S,J)$ be the automorphism group of the zero section of $\sE''(S,J)$.  Then we have an exact sequence
\begin{equation} \label{eqn:15}
0 \rightarrow \sT'' \rightarrow \sE' \rightarrow \sE \rightarrow 0 .
\end{equation}
That is, for each $(S,J)$ we have an exact sequence on the \'etale site of $S$,
\begin{equation*}
0 \rightarrow \sT''(S,J) \rightarrow \sE'(S,J) \rightarrow \sE(S,J) \rightarrow 0 .
\end{equation*}
Provided that $X$ is of Deligne--Mumford type over $Y$, all of the terms appearing in this sequence will be commutative $2$-groups.  One may therefore avoid $3$-groups in this context by taking~\eqref{eqn:15} as the \emph{definition} of a compatible sequence of obstruction theories when $Y$ is locally unobstructed over $Z$ (and $X$ is of Deligne--Mumford type over $Y$).

\section{Method of Proof}\label{sec:method}

The results of this paper could be arranged in a matrix with one column for degenerations and one for pairs and one row each for the maps $\Psi$, $\Theta$, and $\Upsilon$.  The arguments for each of the matrix entries are very similar so we have axiomatized the situation in order not to have to repeat them too many times.  In this section we state two lemmas, one concerning the spaces $\AF$, $\JLi$, $\Kim$, $\ACGS$ and the other about the maps $\Psi$, $\Theta$, and $\Upsilon$.  Using these lemmas we show how to deduce Theorem~\ref{thm:compare}.  Then in sections~\ref{sec:spaces} and~\ref{sec:maps} we give the arguments necessary to prove the lemmas in each case.  At the end of this section, we reduce the statements of Corollaries \ref{cor:GW1} and \ref{cor:GW2} to the statement of our main theorem.

In what follows, we will simplify the notation for our moduli spaces by suppressing the various subscripts for locally constant data.  The reader may imagine either that these data have been fixed, or else that each moduli space is the disjoint union over discrete parameters of moduli spaces with appropriate decorations. 

\subsection{Universal targets}
\notn{$X/V$}{a smooth pair or acceptable degeneration}
\notn{$\sX/\sV$}{the universal smooth pair or acceptable degeneration}
Denote by $X\rightarrow V$ either a smooth pair with ambient space $X$ and $V$ being a point (Definition~\ref{def:sm-pair}) or an acceptable degeneration (Definition~\ref{def:acc-deg}).

A central point is that our main theorems would have been almost immediate if the virtual fundamental classes were simply the fundamental classes.  We are not so fortunate, but our misfortune may be attributed entirely to the global geometry of $X/V$.  It is possible to isolate this global geometry by comparing the space of maps to $X/V$ with the space of  maps to a {\em universal target} $\sX / \sV$.  This method was introduced in~\cite{ACW} and~\cite{ACFW} and has been used in \cite{CMW} as well.  

The universal targets are defined as follows:

\notn{$\sA$}{$[\,\bA^1\,/\,\Gm\,]$, moduli of line bundles with section}
\notn{$\sD$}{$[\,0\,/\,\Gm\,]$, substack of line bundles with zero section in $\sA$}
\begin{enumerate}
\item The universal smooth pair is $(\sA, \sD)$ where $\sA = [\,\bA^1\, /\, \Gm\,]$ and $\sD = [\,0\,/\,\Gm\,]$.  In this case, $\sX = \sA$ and $\sV$ is a point.
\item The universal acceptable degeneration is the multiplication morphism $\sA^2 \to \sA$.  Here $\sX = \sA^2$ and $\sV = \sA$.
\end{enumerate}
For more about the universal smooth pair and the universal acceptable degeneration, see \cite[Sections~2.1 and~2.2]{ACFW}.  

In either case, we have a commutative---but not cartesian---diagram
\begin{equation*} \xymatrix{
X \ar[r] \ar[d] & \sX \ar[d] \\
V \ar[r] & \sV .
} \end{equation*}
In both situations, $X$ is smooth over $\sX \fp_{\sV} V$.

\subsection{Comparison of virtual fundamental classes}
As in \cite{ACW} and \cite{CMW}, we will prove Theorem~\ref{thm:compare} using Costello's comparison technique~\cite[Theorem~5.0.1]{costello}.  Let $p:\Kstab(X/V)\to \Lstab(X/V)$ be one of the maps $\Psi, \Theta, \Upsilon$, where $\Kstab(X/V)$ and $\Lstab(X/V)$ are the relevant moduli spaces.  We construct a cartesian diagram
\begin{equation} \label{eqn:4} \vcenter{\xymatrix{
\Kstab(X/V) \ar[r]^p \ar[d]_f & \Lstab(X/V) \ar[d] \\
K^*(\sX/\sV) \ar[r]^q & L(\sX/\sV).
}} \end{equation}
Here the top arrow is the map $p=\Psi, \Theta,$ or $\Upsilon$ in question.  The spaces $K(\sX/\sV)$ and $L(\sX/\sV)$ are the moduli spaces of pre-stable maps to the universal family $\sX/\sV$, with the respective expanded, logarithmic or orbifold structures;  the star in $K^*(\sX/\sV)$ indicates a small modification to $K(\sX/\sV)$ in the case $p = \Upsilon$, necessary to ensure that the diagram commutes.  The use of such a modification goes back at least as far as Behrend's proof of the product formula for Gromov--Witten invariants~\cite{Beh-prod}.  The technique has also been used in~\cite{costello,ACW,AF,AJT,CMW}.

We reserve the notation  $\fM(X/V)$ and $\fM(\sX/\sV)$ for the usual moduli spaces of pre-stable maps to the given family $X/V$ or to the universal target $\sX/\sV$, without added expanded, logarithmic or orbifold structure:  see section~\ref{sec:naive}.

In section~\ref{sec:spaces} we will describe the spaces $\JLi$, $\AF$, $\Kim$ and $\ACGS$, and prove the following facts about each of them:
\begin{mainlemma} \label{lem:spaces}
For $K \in \set{\AF,\JLi,\ACGS,\Kim}$, 
\begin{enumerate}[label=(\roman{*}), ref=(\roman{*})]
\item \label{spaces:proper-DM} $\Kstab(X/V)$ is a proper,  Deligne--Mumford stack;
\item \label{spaces:artin} $K(\sX/\sV)$ is an Artin stack;
\item \label{spaces:dense} the locus of totally nondegenerate maps (see Definiton~\ref{def:nondeg}) in $K(\sX/\sV)$ is dense;
\item \label{spaces:cart} there is a cartesian diagram 
\begin{equation} \label{diag:naive} \vcenter{\xymatrix{
K(X/V) \ar[r] \ar[d] & \fM(X/V) \ar[d] \\
K(\sX/\sV) \ar[r] & \fM(\sX/\sV);
}} \end{equation}
%\item \label{spaces:obs2} the relative obstruction theory for $K(X/V)$ over $K(\sX/\sV) \fp_{\sV} V$ is the pullback of that for $\fM(X/V)$ over $\fM(\sX/\sV) \fp_{\sV} V$;
\item \label{spaces:obs} there are relative obstruction theories as indicated in the diagram below\footnote{When we need to indicate that these obstruction theories are associated to $K \in \set{\JLi,\AF,\Kim,\ACGS}$ we decorate them with a subscript.}
\begin{equation*}
\overbrace{\lefteqn{\overbrace{\phantom{\Kstab(X/V) \rightarrow K(\sX/\sV) \fp_{\sV}\vphantom{\fp^\sV} V}}^{\sE'}} \Kstab(X/V) \rightarrow \underbrace{K(\sX/\sV) \fp_{\sV} V \rightarrow \sZ}_{\sE''}}^{\sE} ,
\end{equation*}
fitting into an compatible sequence
\begin{equation*}
0 \rightarrow \sE' \rightarrow \sE \rightarrow \sE'' \rightarrow 0
\end{equation*}
in which 
\begin{enumerate*}[label=(\alph{*})]
\item $\sE'$ is pulled back from the relative obstruction theory for $\fM(X/V)$ over $\fM(\sX/\sV) \fp_{\sV} V$ (see section~\ref{sec:naive}) using diagram~\eqref{diag:naive}, and
\item the ``original'' virtual fundamental class of $\Kstab(X/V)$ is that defined using $\sE$;
\end{enumerate*}
\item \label{spaces:unobs} the map $K(\sX/\sV) \fp_{\sV} V \rightarrow \sZ$ is either 
\begin{enumerate}[label=(\alph{*})]
\item \label{unobs:smooth} (for $K = \AF$) locally unobstructed with respect to $\sE''$, or
\item \label{unobs:DM} (for $K \in \set{\JLi, \Kim, \ACGS}$) of Deligne--Mumford type and locally unobstructed with respect to $\sE''$ on the locus of totally nondegenerate maps.
\end{enumerate}
\end{enumerate}
\end{mainlemma}

In section~\ref{sec:maps} we construct the diagrams~\eqref{eqn:4} and verify the basic properties we need:
\begin{mainlemma} \label{lem:maps}
For $p \in \set{\Psi, \Theta, \Upsilon}$,
\begin{enumerate}[label=(\roman{*}), ref=(\roman{*})]
\item \label{maps:cart} diagram~\eqref{eqn:4} is cartesian;
\item \label{maps:etale} the projection from $K^*(\sX/\sV)$ to $K(\sX/\sV)$ is \'etale;
\item \label{maps:DM} $q$ is of Deligne--Mumford type;
\item \label{maps:birat} $q$ is generically an isomorphism;
\item \label{maps:compat} the relative obstruction theory $\sE'_K$ (see Lemma~\ref{lem:spaces}~\ref{spaces:obs}) for $\Kstab(X/V)$ over $K(\sX/\sV)$ is isomorphic to the pullback of $\sE'_L$.
\end{enumerate}
\end{mainlemma}

Before we see how these lemmas imply our theorem, we note the following:
\begin{enumerate}
\item In view of Lemma~\ref{lem:spaces}~\ref{spaces:cart}, Lemma~\ref{lem:maps}~\ref{maps:compat} is true by construction; indeed, $\Kstab(X/V) \subset K(X/V)$ is open, and the obstruction theory on $\Kstab(X/V)$ is the restriction of that of $K(X/V)$.
\item When $p \not= \Upsilon$, Lemma~\ref{lem:spaces}~\ref{spaces:dense} implies Lemma~\ref{lem:maps}~\ref{maps:birat} because in those cases we will take $K^\ast(\sX/\sV) = K(\sX/\sV)$
\Jonathan{fixed typo}
and $q : K(\sX/\sV) \rightarrow L(\sX/\sV)$ restricts to an isomorphism on the dense open substacks of totally nondegenerate maps, by definition;
%\item in all cases, we will define the relative obstruction theory for $K(X/V)$ over $K(\sX/\sV)\times_\sV V$ using the cartesian diagram of Lemma~\ref{lem:spaces}~\ref{spaces:cart}, so Lemma~\ref{lem:spaces}~\ref{spaces:obs2} will be true by definition;\dan{This seems to contradict the statement at the end of ~\ref{lem:spaces}~\ref{spaces:obs}!!}\Jonathan{does the new arrangement alleviate the confusion?}
\item when $p \not= \Upsilon$, Lemma~\ref{lem:maps}~\ref{maps:etale} is obvious since $K^\ast(\sX/\sV) = K(\sX/\sV)$; and
\item Lemma~\ref{lem:maps}~\ref{maps:etale} implies that the obstruction theory for $\Kstab(X/V)$ over $K(\sX/\sV) \fp_{\sV} V$ is also an obstruction theory for $\Kstab(X/V)$ over $K^\ast(\sX/\sV) \fp_{\sV} V$.
\end{enumerate}

Lemmas~\ref{lem:spaces} and~\ref{lem:maps} immediately imply the hypotheses of \cite[Theorem~5.0.1]{costello}.  From Costello's theorem, we may conclude that
\begin{equation*}
p_\ast \Big[ \Kstab(X/V) \Big/ K^\ast(\sX/\sV) \fp_{\sV} V \Big]^{\vir} = \Big[ \Lstab(X/V) \Big/ L(\sX/\sV) \fp_{\sV} V \Big]^{\vir} .
\end{equation*}
Since $K^\ast(\sX/\sV)\to K(\sX/\sV)$ is \'etale we have
\begin{equation*}
\Big[ \Kstab(X/V) \Big/ K(\sX/\sV) \fp_{\sV} V \Big]^{\vir} = \Big[ \Kstab(X/V) \Big/ K^*(\sX/\sV) \fp_{\sV} V \Big]^{\vir},
\end{equation*} 
so this implies that
\begin{equation} \label{eqn:17}
p_\ast \Big[ \Kstab(X/V) \Big/ K(\sX/\sV) \fp_{\sV} V \Big]^{\vir} = \Big[ \Lstab(X/V) \Big/ L(\sX/\sV) \fp_{\sV} V \Big]^{\vir} .
\end{equation}

\begin{remark}
Had the virtual fundamental classes for each moduli space $K \in \set{\JLi,\AF,\Kim,\ACGS}$ been defined using the obstruction theories $\sE'_K$, would now be done.  This is not the case, but the virtual fundamental classes obtained from the obstruction theories $\sE_K$ and $\sE'_K$ are nevertheless the same, as the argument below demonstrates. 
\end{remark}

We now divide the argument according to the cases~\ref{unobs:smooth} and~\ref{unobs:DM} of Lemma~\ref{lem:spaces}~\ref{spaces:unobs}.  In case~\ref{unobs:DM}, there is by Lemma~\ref{lem:spaces}~\ref{spaces:dense} a dense open substack of $K(\sX/\sV)$ on which the relative obstruction theory over $\sZ$ vanishes.  It follows by \cite[Lemma~B.2]{ACW} that $K(\sX/\sV) \fp_{\sV} V$ is a relative local complete intersection over $\sZ$.  The relative virtual fundamental class therefore agrees with the fundamental class.  Now by \cite[Theorem~4.8]{Manolache}, the relative virtual fundmanental class of $K(X/V)$ over $\sZ$ is the virtual pullback of the virtual class of $K(\sX/\sV) \fp_{\sV} V$ over $\sZ$, which is just the fundamental class:
\Jonathan{added $\vir$ superscript and reformatted a little}
\begin{equation} \label{eqn:16}
\Bigl[ \Kstab(X/V) \Big/ \sZ \Bigr]^{\vir} = f^!_{\sE'} \Bigl[ \Kstab(\sX/\sV) \fp_{\sV} V \Big/ \sZ \Big] = \Bigl[ \Kstab(X/V) \Big/ \Kstab(\sX/\sV) \fp_{\sV} V \Bigr]^{\vir} .
\end{equation}

%The stack $K(\sX/\sV) \fp_{\sV} V$ has a dense open substack on which the relative obstruction theory over $\sZ$ vanishes (Lemma~\ref{lem:spaces}~\ref{spaces:unobs}).  It follows from \cite[Lemma~B.2]{ACW}  that $K(\sX/\sV) \fp_{\sV} V$ is a relative local complete intersection over $\sZ$.  Provided that $K(\sX/\sV) \fp_{\sV} V$ is of Deligne--Mumford type over $\sZ$  this implies that the relative virtual fundamental class for $K(\sX/\sV) \fp_{\sV} V$ over $\sZ$ is just the fundamental class.  By \cite[Theorem~4.8]{Manolache}, the relative virtual fundamental class of $K(X/V)$ over $\sZ$ is then the virtual pullback of the relative virtual class of $K(\sX/\sV) \fp_{\sV} V$ over $\sZ$.  

\noindent Unfortunately, virtual fundamental class technology has not yet progressed to the point where we can make the same argument when $K(\sX/\sV) \fp_{\sV} V$ is of Artin type over $\sZ$.  If, however, it is locally unobstructed over $\sZ$ (Lemma~\ref{lem:spaces}~\ref{spaces:unobs}~\ref{unobs:smooth}) we may apply \cite[Proposition~6.5]{obs} to obtain the same conclusion.  

%We therefore have
%\begin{equation*}
%\Big[ \Kstab(X/V) \Big/ K(\sX/\sV) \fp_{\sV} V \Big]^{\vir} = \Big[ \Kstab(X/V) \Big/ \sZ \Big]^{\vir} = \big[ \Kstab(X/V) \big]^{\vir} .
%\end{equation*}
Combining~\eqref{eqn:16} (and the corresponding fact for $L$) with~\eqref{eqn:17},  we get the conclusion of Theorem~\ref{thm:compare}:
\begin{equation*}
p_\ast \big[ \Kstab(X/V) \big] ^{\vir} = \big[ \Lstab(X/V) \big]^{\vir} .
\end{equation*}

\subsection{Comparison of GW invariants}\label{sec:comp-GW}
Corollaries~\ref{cor:GW1}~and~\ref{cor:GW2} reduce to our main comparison of virtual classes by an application of the projection formula.  For $X/V$ a smooth pair with divisor $D$, denote by $\alpha_i$ the prescribed contact order of the $i$-th marked point along the divisor.  Denote by $\Sigma_i^K:\Kstab(X/V)\rightarrow \oC$ the corresponding section of the universal stabilized coarse curve.  In the cases $K\in\set{\JLi, \Kim,\AF}$, we have evaluation maps
\begin{align*}
  e_i^K & :\Kstab(X/V) \rightarrow X && \text{for $\alpha_i = 0$, and} \\ 
  e_i^K & :\Kstab(X/V)\rightarrow D && \text{for $\alpha_i > 0$}.
\end{align*}
(Note that these are the correct targets for the Abramovich--Fantechi evaluation maps because because $X$ and $D$ were taken without orbifold structure.)  When $K=\ACGS$, we have slightly more refined evaluation maps
\begin{align*}
  \tilde e_i^\ACGS & :\ACGSstab(X/V) \rightarrow \wedge_{\alpha_i}X
\end{align*}
taking values in the logarithmic evaluation stack $\wedge_{\alpha_i} X$ (see \cite{ACGM}) of contact order $\alpha_i$.  This stack may be presented as
\begin{align*}
  \wedge_0 X\cong& X\times B\Gm && \text{for $\alpha_i = 0$, and} \\ 
  \wedge_{\alpha_i} X\cong& \bigl[\cO_X(D)|_D\big/\Gm\bigr] && \text{for $\alpha_i > 0$}.
\end{align*}
where the quotient is by the $\alpha_i$-th power of the natural $\Gm$ scaling action. At the expense of losing some logarithmic information, we compose with the natural projections to $X$ and to $D$ we again obtain evaluation maps 
\begin{align*}
  e_i^\ACGS & :\Kstab(X/V) \rightarrow X && \text{for $\alpha_i = 0$, and} \\ 
  e_i^\ACGS & :\Kstab(X/V)\rightarrow D && \text{for $\alpha_i > 0$}.
\end{align*}

In each case, we have descendant classes $\psi_i=c_1\bigl((\Sigma_i^K)^\ast\omega_{\oC/\Kstab}\bigr)$.  Our terms are defined so that
\begin{itemize}
\item $e_i^K=e_i^L\circ p$
\item $p^\ast \omega_{\oC/\Lstab} = \omega_{\oC/\Kstab}$, and
\item $p^\ast \Sigma_i^L = \Sigma_i^K$.
\end{itemize}
The projection formula gives
\begin{equation*}
 \int_{[\Kstab(X/V)]^{\vir}} \prod_i (\psi^K_i)^{a_i}(e^K_i)^\ast\gamma_i =  \int_{p_\ast[\Kstab(X/V)]^{\vir}} \prod_i (\psi_i^L)^{a_i}(e^L_i)^\ast\gamma_i.
\end{equation*}
where the integrals are taken over $\Kstab(X/V)$ and $\Lstab(X/V)$ respectively.
  Thus Theorem~\ref{thm:compare-pairs} implies Corollary~\ref{cor:GW1} and Theorem~\ref{thm:compare-deg} implies Corollary~\ref{cor:GW2}.

\section{Spaces of stable maps} \label{sec:spaces}

\subsection{The naive theory} \label{sec:naive}

\notn{$\fM(X/V)$}{moduli of maps from pre-stable curves to $X/V$}

Let $\fM(X/V)$ be the moduli space of commutative diagrams
\begin{equation} \label{eqn:10} \xymatrix{
C \ar[r] \ar[d] & X \ar[d] \\
S \ar[r] & V
} \end{equation}
in which $C$ is a family of pre-stable curves over $S$.

There is a map $\fM(X/V) \rightarrow \fM(\sX/\sV)$ by composition with the maps $X \rightarrow \sX$ and $V \rightarrow \sV$; there is also a map $\fM(X/V) \rightarrow V$ by forgetting the upper half of diagram~\eqref{eqn:10}.  We construct a relative obstruction theory for $\fM(X/V)$ over $\fM(\sX/\sV) \fp_{\sV} V$, using the formalism introduced in \cite{obs}.  Given a square-zero lifting problem
\begin{equation*} \xymatrix{
S \ar[r] \ar[d] & \fM(X/V) \ar[d] \\
S' \ar@{-->}[ur] \ar[r] & \fM(\sX/\sV) \fp_{\sV} V
} \end{equation*}
we obtain an extension problem
\begin{equation*} \xymatrix{
& & X \ar[d] \\
C \ar[r] \ar[d]_\pi \ar@/^10pt/[urr]^f & C' \ar@{-->}[ur] \ar[r] \ar[d] & \sX \fp_{\sV} V \ar[d] \\
S \ar[r] & S' \ar[r] & V
} \end{equation*}
which translates into the lifting problem
\begin{equation} \label{eqn:5} \vcenter{\xymatrix{
C \ar[r] \ar[d] & X \ar[d] \\
C' \ar@{-->}[ur] \ar[r] & \sX \fp_{\sV} V .
}} \end{equation}
Local lifts form a torsor on $C$ under the sheaf of groups $\uHom(f^\ast \Omega_{X / \sX \fp_{\sV} V}, \pi^\ast J)$, where $J = I_{S/S'}$ is the ideal of $S$ in $S'$.  In both situations, relative and degenerate, we have $\Omega_{X / \sX \fp_{\sV} V} = \Omega_{X/V}^{\log}$:  by Corollary~\ref{cor:log-stacks}, $\sX$ is log.\ \'etale over $\sV$, so $\sX \fp_{\sV} V$ is \'etale over $\Log(V)$; it follows that $\Omega_{X/\sX \fp_{\sV} V} = \Omega_{X/\Log(V)}$, which is $\Omega_{X/V}^{\log}$ by, for example, \cite[proof of~(4.6)]{Olsson_log}.  If we now define $\sF(S,J)$ to be the category of torsors on $C$ under $\uHom(f^\ast \Omega_{X/V}^{\log}, \pi^\ast J) = f^\ast T_{X/V}^{\log} \tensor \pi^\ast J$, then $\sF$ forms a relative obstruction theory for $\fM(X/V)$ over $\fM(\sX/\sV) \fp_{\sV} V$ as in section~\ref{sec:obs}.

\begin{proposition}
$\sF$ is a perfect relative obstruction theory for $\fM(X/V)$ over $\fM(\sX/\sV) \fp_{\sV} V$.
\end{proposition}
\begin{proof}
One argument may be found in \cite[Section~7.3]{obs}.  More directly one can see that $\sF$ is represented by the complex $R \pi_\ast (f^\ast T_{X/V}^{\log})^\vee[1]$, which is perfect in degrees $[-1,0]$ since the fibers of $\pi$ are curves and $T_{X/V}^{\log}$ is a vector bundle.
\end{proof}

\subsubsection*{Totally nondegenerate maps}

\notn{$\Mnd(X/V)$}{totally non-degenerate maps in $\fM(X/V)$}

The relative obstruction theory of $\fM(X/V)$ over $\fM(\sX/\sV)$ is useless by itself.  It can only be used to define virtual curve counting invariants relative to bases that themselves possess virtual fundamental classes;  $\fM(\sX/\sV)$ won't do as a base because its deformation theory is too badly behaved.  The role played alternately by the compactifications $\JLi$, $\AF$, $\Kim$, and $\ACGS$ considered below is to serve as an appropriate base.  

Nevertheless, the deformation theory of $\fM(\sX/\sV)$ is not so bad for \emph{totally nondegenerate objects}:
\begin{definition} \label{def:nondeg}
An object~\eqref{eqn:10} of $\fM(X/V)$ is called \emph{totally nondegenerate} if
\begin{enumerate}[label=(\roman{*})]
\item the source curve $C$ is smooth,
\item (for pairs) the pre-image in $C$ of the special divisor in $X$ is finite over $S$, and
\item (for acceptable degenerations) $X \fp_V S$ is smooth over $S$.
\end{enumerate}
We write $\Mnd(X/V) \subset \fM(X/V)$ for the stack of totally nondegenerate maps to the fibers of $X$ over $V$.
\end{definition}
We will see below that all of the spaces $K(\sX/\sV)$ for $K \in \set{\JLi,\AF,\Kim,\ACGS}$ contain $\Mnd(\sX/\sV)$ as a dense open substack.  
%Note that in each case, there are natural forgetful maps $K(X/V)\rightarrow \fM(X/V)$ and $K(\sX/\sV)\rightarrow \fM(\sX/\sV)$ making diagram~\eqref{diag:naive} commutative.

\subsection{Expanded targets:  the theory of J.\ Li} \label{sec:JLi}

\subsubsection{Pre-deformability and relative stable maps}

\notn{$X^{\exp}/V^{\exp}$}{universal expanded target of $X/V$}
\notn{$\sX^{\exp}/\sV^{\exp}$}{universal expansion of universal target}

Recall from~\cite[Sections~2.1 and~2.3]{ACFW} that there is a moduli space of expansions of the family $\sX / \sV$.  Here we denote the base of this family by $\sV^{\exp}$ and the universal expanion by $\sX^{\exp}$.  The moduli space of expansions of $X/V$ and its universal family are defined by base change:  $V^{\exp} = V \fp_{\sV} \sV^{\exp}$ and $X^{\exp} = X \fp_{\sX} \sX^{\exp}$.

J.\ Li defined a moduli space of stable maps into the fibers of $X^{\exp} / V^{\exp}$~\cite{Li1}.  An object of $\JLi(X / V)$ is a \textit{predeformable} commutative diagram
\begin{equation} \label{eqn:3} \xymatrix{
C \ar[r] \ar[d] & X^{\exp} \ar[d] \\
S \ar[r] & V^{\exp} 
} \end{equation}
(see below for the definition of predeformability).  Diagram~\eqref{eqn:3} is said to be \emph{stable} if its automorphism group over $X \rightarrow V$ is finite. 
\Jonathan{reworded}

We will state the predeformability condition in terms of the $S$-morphism $f : C \rightarrow X^{\exp}_S = X^{\exp} \fp_{V^{\exp}} S$.  We employ the following terminology:
\begin{enumerate}[label=(\roman{*})]
\item the \emph{special locus} of $X^{\exp}_S$ is the union of the non-smooth locus and, in the case of expanded pairs, the distinguished divisor;
\item a node of $C$ is called \emph{essential} or \emph{distinguished} if it is carried by $f$ into the special locus of $X^{\exp}_S$.
\end{enumerate}
Suppose that $s \in S$ is a point and $p$ is an essential node of $C_s$.  The map $C \rightarrow X^{\exp}_S$ is \emph{predeformable} at $p$ if \'etale-locally in $C$ and smooth-locally in $X^{\exp}_S$, it admits the following form:
\begin{equation*} 
A[x,y] / (xy - t) \leftarrow A[u,v] / (uv - w), 
\end{equation*}
in which, $w = t^n$, $u \mapsto x^n$, and $v \mapsto y^n$.  The map is predeformable if it is predeformable at every essential node.

In the case of pairs we have a similar local condition near points mapping to the distinguished divisor:  it should have the form
\begin{equation*}
A[x] \leftarrow A[u]
\end{equation*}
in which $u \mapsto x^n$ for some positive integer $n$ (that remains fixed in families).
\Jonathan{added paragraph}

\begin{definition}
\notn{$\JLi(X/V)$}{pre-stable maps to expansions of $X/V$}
We write $\JLi(X/V)$ for the stack of predeformable maps to the family $X / V$ and $\JListab(X/V)$ for its open substack of stable maps.
\end{definition}

\begin{nlem}[\namer{\ref{lem:spaces}}{$\JLi$}{\ref{spaces:proper-DM}}]
The stack $\JListab(X/V)$ is a proper, Deligne--Mumford stack.
\end{nlem}
\begin{proof}
See \cite[Theorems~0.1 and~0.2]{Li1}.
\end{proof}

\begin{nlem}[\namer{\ref{lem:spaces}}{$\JLi$}{\ref{spaces:artin}}]
The stack $\JLi(\sX/\sV)$ is an Artin stack.
\end{nlem}
\begin{proof}
Using the projection $\JLi(\sX/\sV) \rightarrow \sV^{\exp}$ and the fact that $\sV^{\exp}$ is algebraic \cite[Theorem~1.3.1]{ACFW}, it is sufficient to show that $\JLi(\sX/\sV) \fp_{\sV^{\exp}} U$ is algebraic after base change to a suitable cover $U \rightarrow \sV^{\exp}$.  By \cite[Proposition~8.3.1]{ACFW} there are maps $\sA^n \rightarrow \sV^{\exp}$ forming a cover, so it is sufficient to show that the stacks $\JLi(\sX/\sV) \fp_{\sV^{\exp}} \sA^n$ are algebraic.  Using \cite[Section~8.2]{ACFW} we can reinterpret $\JLi(\sX/\sV)$ as an open substack of the moduli space of commutative diagrams
\begin{equation*} \xymatrix{
C \ar[r] \ar[d]_\pi & (\sA \times \sA)^n \ar[d]^{m^n} \\
S \ar[r] & \sA^n
} \end{equation*}
where the map $m^n : (\sA \times \sA)^n \rightarrow \sA^n$ is defined by multiplication on each of the $n$ factors.  It is sufficient therefore to show that the stack of all such diagrams is algebraic, and to deduce this it is clearly enough to consider the case $n = 1$.  An $S$-point of this stack is a tuple $(C, L_1, L_2, s_1, s_2)$ where 
\begin{enumerate}[label=(\roman{*})]
\item $C$ is a family of curves over $S$, 
\item $L_1$ and $L_2$ are line bundles on $C$,
\item $s_i \in \Gamma(C, L_i)$, and
\item the line bundle and section pair $(L_1 \tensor L_2, s_1 \tensor s_2)$ is pulled back from a pair on $S$.
\end{enumerate}
The first three data are well-known to be parameterized by algebraic stacks over $S$.  We may therefore assume that they are all given, along with a line bundle and section $(M,t)$ on $S$.  The isomorphisms between $\pi^\ast M$ and $L_1 \tensor L_2$ form a $\Gm$-torsor $P$ over $C$ and the locus $Q \subset P$ parameterizing those maps carrying $\pi^\ast t$ to $s_1 \tensor s_2$ is a closed subscheme, and is in particular affine over $C$.  Since $C$ is proper over $S$, the sheaf $\pi_\ast Q$ is representable by an algebraic space and we are done.
\end{proof}

\subsubsection{Deformation and obstruction theory}

We begin with the relative obstruction theory over $\JLi(\sX/\sV) \fp_{\sV} V$.

\begin{nlem}[\namer{\ref{lem:spaces}}{$\JLi$}{\ref{spaces:cart}}]
The diagram \eqref{diag:naive} for $K=\JLi$ is cartesian.
\end{nlem}

\begin{proof}
It suffices to note that a diagram
\begin{equation*} \xymatrix{
C \ar[r] \ar[d] & X^{\exp} \ar[d] \\
S \ar[r] & V^{\exp}
} \end{equation*}
is predeformable if and only if the same property holds of the diagram
\begin{equation*} \xymatrix{
C \ar[r] \ar[d] & \sX^{\exp} \ar[d] \\
S \ar[r] & \sV^{\exp}
} \end{equation*}
induced by composition.
\end{proof}

It follows that the pullback of the relative obstruction theory $\sF$ (section~\ref{sec:naive}) gives a relative obstruction theory $\sE'$ for $\JLi(X/V)$ over $\JLi(\sX/\sV) \fp_{\sV} V$.

\subsubsection{The obstruction theory relative to the moduli space of curves}

In \cite{Li2}, Li described a relative obstruction theory for $\JLi(X/V)$ over the moduli stack of prestable curves $\fM$. To be more precise, Li describes an absolute obstruction theory, which is written explicitly as an extension of a relative obstruction theory by the tangent space of $\fM$.  We summarize Li's definition, in somewhat different terms, as follows.  A more detailed discussion of the issues involved appears in~\cite{CMW}.  

Consider a lifting problem,
\begin{equation*} \xymatrix{
S \ar[r] \ar[d] & \JLi(X/V) \ar[d] \\
S' \ar[r] \ar@{-->}[ur] & \fM ,
} \end{equation*}
in which $S'$ is a square-zero extension of $S$ by an ideal $J$.  This corresponds to an extension problem,
\begin{equation} \label{eqn:11} \xymatrix{
C \ar[r] \ar[d] \ar@/^15pt/[rr] & C' \ar@{-->}[r] \ar[d] & X^{\exp} \ar[d] \\
S \ar[r] \ar@/_15pt/[rr] & S' \ar@{-->}[r] & V^{\exp},
} \end{equation} 
\vskip2mm
\noindent where the right square is required to be {\em predeformable}.
Li shows that this problem has a solution, \emph{provided one is allowed to work \'etale locally in both $S$ and in $C$} \cite[Lemma~1.12]{Li2}.  If we define a site $\et(C/S)$ whose objects are commutative diagrams
\begin{equation*} \xymatrix{
U \ar[r] \ar[d] & C \ar[d] \\
V \ar[r] & S
} \end{equation*}
in which the horizontal arrows are \'etale (and families are covering if both collections of horizontal arrows are covering), then \cite[Lemma~1.12]{Li2} says that the deformation problem~\eqref{eqn:11} has a predeformable solution locally in $\et(C/S)$.  It follows formally that there is an abelian $2$-group $T(S, J)$ on $\et(C/S)$, depending on a quasi-coherent $\cO_S$-module $J$, such that the predeformable solutions to the lifting problem~\eqref{eqn:11} form a torsor under $T(S,J)$ (see section~\ref{sec:con-obs}).  If we define $\sE(S,J)$ to be the category of torsors on $\et(C/S)$ under $T(S,J)$ then $\sE$ forms a relative obstruction theory for $\JLi(X/V)$ over $\fM$.

\begin{lemma} \label{lem:jli-perf}
$\sE$ is a perfect relative obstruction theory.
\end{lemma}
\begin{proof}
See \cite[Section~1.2]{Li2}.\Jonathan{this is a copout; I'll put a proper reference to \cite{obs} when it's ready}
\end{proof}

Let $\sE'$ denote the pullback of the relative obstruction theory $\sF$ (section~\ref{sec:naive})  to $\JLi(X/V)$; let $\sE''$ denote the relative obstruction theory for $\JLi(\sX/\sV)$ over $\fM$ obtained by applying the above construction in the case $X = \sX$ and $V = \sV$.  By definition, $\sE''(S,J)$ is the category of torsors under a sheaf of groups $T''(S,J)$ on $\et(C/S)$.

%\begin{proposition} \label{prop:Li-compat}
\begin{nlem}[\namer{\ref{lem:spaces}}{$\JLi$}{\ref{spaces:obs}}]  \label{lem:Li-compat}
There is a compatible sequence of obstruction theories
\begin{equation*} 
0 \rightarrow \sE' \rightarrow \sE \rightarrow \sE'' \rightarrow 0 .
\end{equation*}
for the sequence of maps
\begin{equation*}
\JLi(X/V) \rightarrow \JLi(\sX/\sV) \fp_{\sV} V \rightarrow \fM .
\end{equation*}
In the notation of \hyperref[spaces:obs]{Lemma~\ref*{lem:spaces}~\ref*{spaces:obs}}, $\sZ = \fM$.
\end{nlem}

Before we give the proof, we give an alternate construction of $\sE'$ that will be more easily comparable to $\sE$.  Note that the lifting problem~\eqref{eqn:5} also gives rise to a torsor of lifts on the site $\et(C/S)$ under the sheaf of groups $T'$ whose value on $U/V$ is $\Hom_U(f^\ast \Omega_{X/V}^{\log} \rest{U}, \pi^\ast J \rest{U})$.  Therefore we have another obstruction theory, $\sF$, with $\sF(S,J)$ being the category of torsors on $\et(C/S)$ under $T'$.

\begin{lemma} \label{lem:E=F}
$\sE' = \sF$ .
\end{lemma}

Before proving this, we need a lemma about the site $\et(C/S)$.  There is an embedding of sites $i : \et(C) \rightarrow \et(C/S)$ given by the functor $i_!(U) = U/S$.  The right adjoint, $i^\ast$, of $i_!$ is given by $i^\ast(U/V) = U$.  If $F$ is a sheaf on $\et(C)$ then $i_\ast(F)(U/V) = F(U)$.  Thus $T' = i_\ast \uHom(f^\ast \Omega_{X/V}^{\log}, \pi^\ast J)$.

\begin{lemma} \label{lem:i-ex}
$i_\ast$ is exact.
\end{lemma}
\begin{proof}
Since pushforward is left exact, we only need to show that it preserves local surjections of presheaves, and if $F \rightarrow F'$ is a locally surjective morphism of presheaves on $\et(C)$ then it is clear from the topology on $\et(C/S)$ that it is also locally surjective on $\et(C/S)$.\footnote{Another way to see the exactness is to remark that $\et(C) \subset \et(C/S)$ is the closed embedding complementary to the open embedding $\et(\varnothing / S) \subset \et(C/S)$, and pushforward is exact for a closed embedding.}
\end{proof}

\begin{proof}[Proof of Lemma~\ref{lem:E=F}]
It follows from Lemma~\ref{lem:i-ex} that the natural maps 
\begin{equation*}
H^i(\et(C/S), T') \rightarrow H^i(\et(C), \uHom(f^\ast \Omega_{X/V}^{\log}, \pi^\ast J))
\end{equation*}
are isomorphisms for all $i$.  Therefore the map $\sF \rightarrow \sE'$ is an equivalence.
\end{proof}

\begin{lemma} \label{lem:T-ex}
There is an exact sequence of sheaves of groups
\begin{equation*}
0 \rightarrow T' \rightarrow T \rightarrow T'' \rightarrow 0
\end{equation*}
on $\et(C/S)$.
\end{lemma}
\begin{proof}
This will be immediate once we recall the definitions of the sheaves in question.  Note first that
\begin{align*}
T'(U/V) & = \left\{ \text{ lifts } \vcenter{\xymatrix{
U \ar[r] \ar[d] & X \ar[d] \\
U[\pi^\ast J] \ar[r]^<>(0.5)0 \ar@{-->}[ur] & \sX \fp_{\sV} V
}} \right\} 
 = \left\{ \text{ lifts } \vcenter{\xymatrix{
U \ar[r] \ar[d] & X^{\exp} \ar[d] \\
U[\pi^\ast J] \ar[r]^<>(0.5)0 \ar@{-->}[ur] & \sX^{\exp} \fp_{\sV^{\exp}} V^{\exp}
}} \right\} .
\end{align*}
We can therefore identify
\begin{align*} T'(U/V) & = 
\left\{ \text{ lifts of } \quad \vcenter{\xymatrix{
& & X^{\exp} \ar[d] \\
U \ar[r] \ar[d] \ar@/^10pt/[urr] & U[ \pi^\ast J \rest{U} ] \ar@{-->}[ur] \ar[r]^<>(0.5)0 \ar[d] & \sX^{\exp} \fp_{\sV} V \ar[d] \\
V \ar[r] & V[ J \rest{V} ] \ar[r]^0 & V^{\exp} 
}} \right\} \\ \\
T(U/V) & = \left\{ \begin{array}{c} \text{ predeformable } \\ \text{lifts of } \end{array} \qquad \vcenter{\xymatrix{
U \ar[r] \ar[d] \ar@/^15pt/[rr] & U[\pi^\ast J \rest{U}] \ar@{-->}[r] \ar[d] & X^{\exp} \ar[d] \\
V \ar[r] \ar@/_15pt/[rr] & V[J \rest{U}] \ar@{-->}[r] & V^{\exp}
}} \right\} \\
\\
T''(U/V) & = \left\{   \begin{array}{c} \text{predeformable} \\ \text{lifts of} \end{array} \vcenter{\xymatrix{
U \ar[r] \ar[d] \ar@/^15pt/[rr] & U[\pi^\ast J \rest{U}] \ar@{-->}[r] \ar[d] & \sX^{\exp} \fp_{\sV} V \ar[d] \\
V \ar[r] \ar@/_15pt/[rr] & V[J \rest{U}] \ar@{-->}[r] & V^{\exp}
}} \right\}
\end{align*}
\vskip 2mm
\noindent The maps in the sequence, as well as its left exactness, are now clear.  To prove the right exactness, we must show that any section of $T''(U/V)$ lifts locally in $\et(C/S)$ to $T(U/V)$.  This is immediate from the smoothness of the map $X^{\exp} \rightarrow \sX^{\exp} \fp_{\sV} V$.
\end{proof}

\begin{proof}[Proof of {\hyperref[lem:Li-compat]{Lemma~\namer{\ref*{lem:spaces}}{$\JLi$}{\ref*{spaces:obs}}}}]
Passing to the categories of torsors under the sheaves of groups in the exact sequence of Lemma~\ref{lem:T-ex}, we get the left exactness of a sequence
\begin{equation*}
0 \rightarrow \sE' \rightarrow \sE \rightarrow \sE'' \rightarrow 0,
\end{equation*} and it remains to show surjectivity on the right.
Obstructions to lifting a $T''$-torsor to a $T$-torsor lie in $H^2(\et(C/S), T')$, which by Lemma~\ref{lem:i-ex} may be identified with the group $H^2(\et(C), \uHom(f^\ast \Omega_{X/V}^{\log}, \pi^\ast J))$.  This vanishes locally in $S$ because $\uHom(f^\ast \Omega_{X/V}^{\log}, \pi^\ast J)$ is quasi-coherent and $C$ is $1$-dimensional over $S$.
\end{proof}

\subsubsection{Nondegenerate maps}

In order to obtain a virtual fundamental class from the relative obstruction theory constructed in the last section, we need a virtual fundamental class on $\JLi(\sX/\sV)$.  

\begin{nlem}[\namer{\ref{lem:spaces}}{$\JLi$}{\ref{spaces:unobs}}]
The stack $\Mnd(\sX/\sV) \subset \JLi(\sX/\sV)$ (Definition~\ref{def:nondeg}) is locally unobstructed (section~\ref{sec:loc-unobs}) with respect to the obstruction theory $\sE''$ .
\end{nlem}
\begin{proof}
A totally nondegenerate $S$-point of $\JLi(\sX/\sV)$ is a commutative diagram
\begin{equation*} \xymatrix{
C \ar[r] \ar[d] & \sX^{\exp} \ar[d] \\
S \ar[r] & \sV^{\exp}
} \end{equation*}
such that the map $S \rightarrow \sV^{\exp}$ factors through the open point of $\sV^{\exp}$.  By the definition of $T$, this implies that if $U/V$ is in $\et(C/S)$ then $T(U/V)$ depends only on $U$.  That is, $T = i_\ast f^\ast T_{\sX}$ where $f : C \rightarrow \sX$ is the structural map and $i : \et(C) \rightarrow \et(C/S)$ is the inclusion of sites.  Since $i_\ast$ is exact (Lemma~\ref{lem:i-ex}), we may conclude that Li's obstruction theory for totally nondegenerate maps coincides with the usual obstruction theory for stable maps.\footnote{This conclusion holds for any $X/V$, not just the universal example.}

To complete the proof, we inspect the stable maps obstruction theories.  For acceptable degenerations, the totally nondegenerate case is completely trivial:  the natural projection $\sX_S \rightarrow S$ is an isomorphism so a totally nondegenerate map is nothing but a Deligne--Mumford--Knudsen pre-stable marked curve, with the obstruction theory being that of the moduli space of pre-stable maps to a point.  In particular, the moduli space is smooth and unobstructed.

For pairs, we have $\sX_S = \sA \times S$ in the totally nondegenerate case and the moduli space can be identified with the moduli space of pre-stable curves equipped with an effective divsor.  Again, this is a smooth stack.  Note that if $\sA$ is viewed as the moduli space of pairs $(L, s)$ where $L$ is a line bundle and $s$ is a section of $L$ then the tangent bundle of $\sA$ (which we notate $T_{\sA}$) can be represented by the complex $\bF = [\cO \xrightarrow{s} L]$, concentrated in degrees $[-1,0]$ (see Proposition~\ref{prop:TsA} in the appendix).  If $J$ is a quasi-coherent sheaf on $S$ then the value on $(S,J)$ of the relative obstruction theory for $\fM(\sA)$ over $\fM$ is the category of torsors on $C$ under $f^\ast T_{\sA} \tensor J$.  The isomorphism classes of these torsors can be identified with sections of $\RR^1 \pi_\ast (\bF \tensor J)$.  Using cohomology and base change, we can deduce that $\RR^1 \pi_\ast (\bF \tensor J) = 0$ by showing it vanishes on the fibers of $C$ over $S$.  We can therefore assume that $S$ is a point.  Now $f^\ast s : \cO_C \rightarrow f^\ast L$ is generically an isomorphism because $f$ is totally nondegenerate so $f^\ast \bF$ is quasi-isomorphic to the cokernel of $f^\ast s$, which is supported on a finite subscheme of $C$.  It follows that $H^1(C, \bF \tensor J) = 0$.\footnote{This also shows that there are no infinitesimal automorphisms, since $H^{-1}(C, \bF \tensor J) = 0$.}
\end{proof}

To complete the construction of the virtual fundamental class, we therefore only need to check that totally nondegenerate maps are dense in $\JLi(\sX/\sV)$.

\begin{nlem}[\namer{\ref{lem:spaces}}{$\JLi$}{\ref{spaces:dense}}]
The totally nondegenerate objects in $\JLi(\sX/\sV)$ form a dense open substack.
\end{nlem}
\begin{proof}
This is a consequence of either
\Jonathan{changed reference below from deleted lemma to AF}
\begin{enumerate*}[label=(\alph{*})]
\item \hyperref[lem:af-dense]{Lemma~\namer{\ref*{lem:spaces}}{$\AF$}{\ref*{spaces:dense}}} and \cite[Lemmas~1.3.1 and~1.4.11]{AF} (which proves the surjectivity of $\AF(\sX/\sV) \rightarrow \JLi(\sX/\sV)$),
\emph{or} 
\item \hyperref[cor:kim-dense]{Lemma~\namer{\ref*{lem:spaces}}{$\Kim$}{\ref*{spaces:dense}}} and 
%\hyperref[lem:kim-li-dm]{Lemma~\namer{\ref*{lem:maps}}{$\Theta$}{\ref*{maps:DM}}}%
\hyperref[lem:kim-li-surj]{Lemma~\ref*{lem:kim-li-surj}}%
%{Lemma~\namer{\ref*{lem:maps}}{$\Psi$}{\ref*{maps:DM}}}%
,
\end{enumerate*}
which are proved below.  In either case, the two lemmas provide a {\em surjective} map $K(\sX/\sV) \rightarrow \JLi(\sX/\sV)$ such that $\Mnd(\sX/\sV) \subset \JLi(\sX/\sV)$ is the image of the dense open substack of $\Mnd(\sX/\sV) \subset K(\sX/\sV)$.
\end{proof}

\subsection{Twisted expansions:  the theory of Abramovich and Fantechi}

\notn{$\bfr$}{twisting choice}

Abramovich and Fantechi do not define a single moduli space of stable maps to expanded targets, but rather an infinite collection of moduli spaces $\AF(X/V)_{\bfr}$, one for each \emph{twisting choice} $\bfr$.  A twisting choice is a function
\begin{equation*}
\bfr : \bDelta \rightarrow \bZ_{>0} ,
\end{equation*}
where $\bDelta$ is the set of all finite multisets of positive integers, such that if $\{c_1,\ldots,c_k\} =  \bc \in \bDelta$ then each $c_i$ divides $\bfr(\bc)$ \cite[Definition~3.4.1]{AF}.  One can always take $\bfr = \lcm(c_i)$---this is known as the \emph{minimal twisting choice}---and it is natural to define $\AF(X/V)$ to be the associated moduli space, but in fact it is no more difficult to prove our comparison theorem for any $\bfr$.\footnote{One reason for considering non-minimal twisting choices is that when $\bfr$ is sufficiently large and divisible, $\AFstab_{g=0}(X/V)_{\bfr}$ is closely related to the moduli space of genus zero orbifold stable maps to a slight modification of $X/V$ (see \cite{ACW}).  This relationship is much less direct for the minimal twisting choice.}

\begin{definition} \label{def:AF} \notn{$\AF(X/V)$}{pre-stable maps to twisted expansions of $X/V$}
Let $\bfr$ be a twisting choice.  Let $\AF(X/V)_{\bfr}$ be the stack whose $S$-points are commutative diagrams
\begin{equation*} \xymatrix{
\tC \ar[rr] \ar[dr] & & \tX \ar[dl] \\
& S
} \end{equation*}
satisfying the following conditions:
\begin{enumerate}
\item $\tC/S$ is a family of twisted curves \cite{AV};
\item $\tX/S$ is a family of twisted expansions of $X/V$ \cite[Section~2.4]{ACFW};
\item $\tC \rightarrow \tX$ is representable and transverse to the special locus;
\item the pre-image of the smooth locus of $\tX / S$ contains the smooth locus of $\tC/S$;
\item for any geometric point $s$ of $S$, and any connected component $D$ of the special locus of $\tX_s$, take $\bc$ to be the multiset of contact orders of the map $\tC \rightarrow \tX_s$ along $D$; then $\bfr(\bc)$ is the order of twisting of $\tX_s$ along $D$.\footnote{With this choice of twisting of $D$, transversality is the same as predefomability.  A common generalization of the definitions of Li and Abramovich--Fantechi may be obtained by dropping the requirement that $c_i$ divide $\bfr(\bc)$ but retaining the predeformability condition in (3):  Li's moduli space then corresponds to the twisting choice in which $\bfr$ is identically $1$.  Note however that Abramovich and Fantechi's obstruction theory does not generalize to $\AF(X/V)_{\boldsymbol{1}}$, and that Li's obstruction theory, when adapted to $\AF(X/V)_{\bfr}$ does not obviously agree with Abramovich and Fantechi's.}
\end{enumerate}
An object of $\AF(X/V)_{\bfr}$ is called \emph{stable} if its automorphism group is finite.  We write $\AFstab(X/V)_{\bfr}$ for the open substack of stable objects of $\AF(X/V)_{\bfr}$.
\end{definition}

\subsubsection{Twisted expanded targets}

\notn{$X^{\bfr} / V^{\bfr}$}{universal $\bfr$-twisted expansion of $X/V$}
\notn{$\sX^{\bfr} / \sV^{\bfr}$}{universal $\bfr$-twisted expansion of universal target}

The analogues of $\sX^{\exp}$ and $\sV^{\exp}$ from the theory of untwisted expansions are more complicated to describe in the twisted theory.  We summarize their constructions below and refer the reader to \cite[Section~7]{ACFW} for a more thorough treatment.

Define a labelled twisted expansion of $X / V$ over $S$ to be a twisted expansion $\tX$ of $X / V$ over $S$, together with a locally constant function, called the \emph{label}, from the special locus of $\tX$ over $S$ to $\bDelta$.  Note that any $S$-point of $\AF(X/V)$ (and, indeed, any $S$-point of $\JLi(X/V)$) gives rise to a labelled twisted expansion of $S$ in which the labels are the orders of contact of the map along the singular locus (and distinguished divisor, if there is one).  

If $\bfr$ is a twisting choice, a labelled twisted expansion $\tX / S$ of $X / V$ with label $\bc$, is called $\bfr$-twisted if for each geometric point $x$ of the special locus of $\tX / S$, the order of twisting of $\tX$ at $x$ is $\bfr(\bc(x))$.  In order for the order of twisting to make sense even when $X \rightarrow V$ is not representable, we take it to mean the order of the relative automorphism group of $x$ in $\tX$ over $X \fp_V S$.

We write $V^{\bfr}$ for the stack of all $\bfr$-twisted expansions of $X/V$ and $X^{\bfr}$ for the universal $\bfr$-twisted expansion.  We have $X^{\bfr} = \sX^{\bfr} \fp_{\sX} X$ and $V^{\bfr} = \sV^{\bfr} \fp_{\sV} V$.

\begin{nlem}[\namer{\ref{lem:spaces}}{$\AF$}{\ref{spaces:proper-DM}}]
The stack $\AFstab(X/V)$ is a proper, Deligne--Mumford stack.
\end{nlem}
\begin{proof}
See \cite[Section~3.3]{AF}.
\end{proof}

\begin{nlem}[\namer{\ref{lem:spaces}}{$\AF$}{\ref{spaces:artin}}]
The stack $\AF(\sX/\sV)$ is an Artin stack.
\end{nlem}
\begin{proof}
See \cite[Lemma~C.1.5]{AF}.
\end{proof}

\subsubsection{The relative obstruction theory}

%\paragraph{Obstruction theory relative to the universal case} 

\begin{nlem}[\namer{\ref{lem:spaces}}{$\AF$}{\ref{spaces:cart}}]
The diagram \eqref{diag:naive} for $K=\AF$ is cartesian.
\end{nlem}
\begin{proof}
Immediate from the fact that the diagrams
\begin{equation*} \vcenter{\xymatrix{
X^{\bfr} \ar[r] \ar[d] & X \ar[d] \\
\sX^{\bfr} \ar[r] & \sX
}} \qquad \text{and} \qquad \vcenter{\xymatrix{
V^{\bfr} \ar[r] \ar[d] & V \ar[d] \\
\sV^{\bfr} \ar[r] & \sV
}} \end{equation*}
are cartesian.
\end{proof}

It follows that the relative obstruction theory for $\fM(X/V)$ over $\fM(\sX/\sV) \fp_{\sV} V$ pulls back to a relative obstruction theory $\sE'$ for $\AF(X/V)$ over $\AF(\sX/\sV) \fp_{\sV} V$.

\subsubsection{The virtual fundamental class} \label{sec:AF-vfc}

Now we recall the obstruction theory $\sE$ for $\AF(X/V)$ over $V^{\bfr}$ from \cite[Section~C.2]{AF}.  Note that the following constructions are only reasonable for transversal maps.  Consider the lifting problem
\begin{equation*} \xymatrix{
S \ar[r] \ar[d] & \AF(X/V) \ar[d] \\
S' \ar@{-->}[ur] \ar[r] & V^{\bfr} .
} \end{equation*}
This translates into
\begin{equation} \label{eqn:6} \vcenter{\xymatrix{
\tC \ar@{-->}[r] \ar[d]_\pi \ar@/^15pt/[rr] & \tC' \ar@{-->}[r] \ar@{-->}[d] & X^{\bfr} \ar[d] \\
S \ar[r] & S' \ar[r] & V^{\bfr} .
}} \end{equation}
Because the map $\tC \rightarrow X^{\bfr}$ is transverse to the singularities, a completion of this diagram exists locally in $\tC$.  Over $\tC$ there is a stack of abelian $2$-groups, $T(S,J)$, defined to be the category of completions of diagram~\eqref{eqn:6} when $S'$ is the trivial square-zero extension of $S$ by $J$ over $V^{\bfr}$.  Since solutions to the deformation problem~\eqref{eqn:6} exist locally in $\tC$, they form a torsor under $T(S,J)$.  If $\sE(S,J)$ is now defined to be the category of torsors on $\tC$ under $T(S,J)$ then $\sE$ forms a relative obstruction theory for $\AF(X/V)$ over $V^{\bfr}$.  Note that the restriction of $\sE(S,J)$ to $\AFstab(X/V)$ is a commutative $2$-group, but without a stability assumption it will be a $3$-group.

The same construction gives an obstruction theory $\sE''$ for $\AF(\sX/\sV)$ over $\sV^{\bfr}$, and by base change, an obstruction theory for $\AF(\sX/\sV) \fp_{\sV} V$ over $V^{\bfr}$.  There is no stability condition in this case, but it is still possible to avoid the use of $3$-groups with the rotation trick discussed in section~\ref{sec:loc-unobs}, combined with the following lemma.

\begin{nlem}[\namer{\ref{lem:spaces}}{$\AF$}{\ref{spaces:unobs}}] \label{lem:AF-unobs}
The stack $\AF(\sX/\sV)$ is locally unobstructed over $\sV^{\bfr}$ with respect to the obstruction theory defined above.
\end{nlem}
\begin{proof}
We show that the obstruction theory described above is trivial.  Consider the extension problem
\begin{equation*} \xymatrix{
\tC \ar@{-->}[r] \ar@/^15pt/[rr] \ar[d]^\pi & \tC[\pi^\ast J] \ar@{-->}[r] \ar@{-->}[d] & \sX^{\bfr} \ar[d] \\
S \ar[r] & S[J] \ar[r] & \sV^{\bfr} .
} \end{equation*}
Solutions to this problem form an abelian $2$-group $T$ over $\tC$.  Let $T'$ be the sheaf of groups obtained by sheafifying the presheaf of isomorphism classes in $T$ and let $T''$ be the kernel of the map $T \rightarrow T'$.  Since $\sX^{\bfr}$ is \'etale over $\sV^{\bfr}$ away from the non-smooth locus, a section of $T$ over the smooth locus of $\tC$ is determined by the curve $\tC'$;  as any two deformations of a smooth curve are locally isomorphic, this means that $T'$ is supported on the nodes of $\tC$.  In particular, $H^1(\tC, T')$ vanishes.

By the long exact sequence in cohomology, any $T$-torsor on $\tC$ is therefore induced from a $T''$-torsor.  But $T''$ is a gerbe over $\tC$:  any two sections of $T''$ are locally isomorphic.  Therefore $T'' = BU$ is the stack of $U$-torsors, where $U$ is the sheaf of automorphisms of the identity section of $T''$.  Then the isomorphism classes of $T''$-torsors can be identified with $H^2(\tC, U)$, which vanishes since $U$ is quasi-coherent and $\tC$ is a curve.
\end{proof}

\begin{nlem}[\namer{\ref{lem:spaces}}{$\AF$}{\ref{spaces:dense}}]\label{lem:af-dense}
The totally nondegenerate objects in $\AF(\sX/\sV)$ form a dense open substack.
\end{nlem}
\begin{proof}
As $\AF(\sX/\sV)$ is smooth over $\sV^{\bfr}$ (since it is locally unobstructed over $\sV^{\bfr}$), the pre-image of a dense open substack of $\sV^{\bfr}$ is dense in $\AF(\sX/\sV)$.  Applying this to the open point of $\sV^{\bfr}$, we find that the locus of maps to unexpanded targets in $\AF(\sX/\sV)$ is dense.  It is immediate from the smoothness of the moduli space of marked curves
\Jonathan{fixed word order problem}
that the totally nondegenerate maps are dense in the locus of maps to unexpanded targets.  
\end{proof}

\begin{lemma}
$\sE$ is a perfect relative obstruction theory for $\AF(X/V)$ over $V^{\bfr}$.\footnote{In fact, if we make use of the trick from section~\ref{sec:loc-unobs} and \hyperref[lem:AF-unobs]{Lemma~\ref*{lem:spaces}~\ref*{spaces:unobs}}, we will only need this lemma for $\AFstab(X/V)$ over $V^{\bfr}$, over which  $\sE$ takes values in commutative $2$-groups.}
\end{lemma}
\begin{proof}
We note that $\sE(S,J)$ is the commutative $3$-group associated to the $3$-term complex $\RR \Hom(G[-1], \pi^\ast J)$ where $G$ is the cone of the morphism of cotangent complexes, $f^\ast \LL_{X^{\bfr} / V^{\bfr}} \rightarrow \LL_{\tC/S}$.  Hence $\sE$ is representable by the complex $R \pi_\ast (G^\vee[1])^\vee$, which is perfect in degrees $[-1,1]$ becaues $G[-1]$ is perfect in degrees $[0,1]$ and the category of quasi-coherent $\cO_{\tC}$-modules has cohomological dimension~$1$.
\end{proof}

\begin{nlem}[\namer{\ref{lem:spaces}}{$\AF$}{\ref{spaces:obs}}]
The obstruction theories defined above fit into a compatible sequence for the maps 
\begin{equation*}
\AF(X/V) \rightarrow \AF(\sX/\sV) \fp_{\sV} V \rightarrow V^{\bfr} .
\end{equation*}
In the notation of \hyperref[spaces:obs]{Lemma~\ref*{lem:spaces}~\ref*{spaces:obs}}, $\sZ = V^{\bfr}$.
\end{nlem}
\begin{proof}
Consider a commutative diagram
\begin{equation*} \xymatrix{
S \ar[r] \ar[d] & \AF(X/V) \ar[d] \\
S' \ar[r] & V^{\bfr}
} \end{equation*}
with $S' = S[J]$.  Let $T$ be the abelian group of lifts of diagram~\eqref{eqn:5} (with $C$ and $C'$ replaced respectively by $\tC$ and $\tC'$); let $T'$ be the abelian $2$-group of completions of~\eqref{eqn:6}; and let $T''$ be the abelian $2$-group of completions of~\eqref{eqn:6} with $X$ replaced by $\sX$ and $V$ replaced by $\sV$.  Then we have an exact sequence
\begin{equation*}
0 \rightarrow T \rightarrow T' \rightarrow T'' \rightarrow 0 .
\end{equation*}
As in the proof of Lemma~\ref{lem:T-ex}, the sequence is left exact by definition and the map $T' \rightarrow T''$ is surjective because $X^{\bfr} \rightarrow \sX^{\bfr} \fp_{\sV} V$ is smooth.  By pushforward, this gives us the left exactness of the sequence of groups of torsors, 
\begin{equation*}
0 \rightarrow \sE'(S,J) \rightarrow \sE(S,J) \rightarrow \sE''(S,J) \rightarrow 0 .
\end{equation*}
The exactness on the right comes from the vanishing of $R^2 \pi_\ast T$, which holds because $C$ is a curve and $T$ is quasi-coherent.

This shows that the obstruction theories are compatible.%
\footnote{To demonstrate directly that the obstruction theories are compatible in the sense of the equivalent definition of section~\ref{sec:loc-unobs} we note that we tautologically have a left exact sequence
\begin{equation*}
0 \rightarrow \pi_\ast T'' \rightarrow \sE' \rightarrow \sE \rightarrow 0.
\end{equation*}
The right exactness is precisely \hyperref[lem:AF-unobs]{Lemma~\ref*{lem:spaces}~\ref*{spaces:unobs}}.}
\end{proof}

\subsection{Logarithmic expansions:  the theory of B.\ Kim}\steffen{This seems to be the only section where we explicitly write log schemes as pairs with their log structures...  shouldn't we decide on a convention and stick with it throughout the rest of the paper?  I am happy not writing them, but that is not a strong opinion.  Our decision should probably be explained to the reader in section 1 or 2.}
\label{sec:kim}

Li's moduli spaces of expanded targets, as well as their universal expansions, may be equipped with logarithmic structures.  There are a number of ways to see this.  Here are two:
\begin{enumerate}
\item $\sV^{\exp}$ can be viewed as the moduli space of \emph{aligned logarithmic structures} and $\sX^{\exp}$ an open substack of the stack of all pairs of logarithmic structures with a common alignment \cite[Section~8.1]{ACFW};
\item $\sV^{\exp}$ is an open substack of the moduli space of $3$-pointed curves and $\sX^{\exp}$ is the quotient of the universal curve by a canonical $\Gm$-action, with respect to which the canonical logarithmic structure is equivariant:  see \cite[Sections~3.1 and~3.3]{ACFW}.
\end{enumerate}
We denote these logarithmic stacks by $\sV^{\exp}_{\log}$ and $\sX^{\exp}_{\log}$.  We use the notation $V^{\exp}_{\log} = (V^{\exp}, M_{V^{\exp}})$ and $X^{\exp}_{\log} = (X^{\exp}, M_{X^{\exp}})$
\Jonathan{added notation with monoids}
for expansions of a family $X/V$.  The maps $X^{\exp}_{\log} \rightarrow \sX^{\exp}_{\log}$ and $V^{\exp}_{\log} \rightarrow \sV^{\exp}_{\log}$ are strict.

%Kim demonstrated \cite[Main Theorem A]{Kim} that there is an algebraic stack $\Kim(\sX/\sV)$ parameterizing \emph{minimal} pre-stable logarithmic maps from logarithmic curves into the fibers of $\sX^{\exp} / \sV^{\exp}$ (which gives Lemma~\ref{lem:spaces}~\ref{spaces:artin}).  Loc.\ cit.\ also provides a proper Deligne--Mumford stack of stable logarithmic maps from logarithmic curves into the fibers of $X^{\exp} / V^{\exp}$ (giving Lemma~\ref{lem:spaces}~\ref{spaces:proper-DM}).

\begin{definition}[{\cite[Sections~5.2.2 and~6.3]{Kim}}] \label{def:kim}
\notn{$\Kim(X/V)$}{log.\ pre-stable maps to log.\ expansions of $X/V$}
Let $\Kim(X/V)$ be the stack whose $S$-point are logarithmically commutative diagrams
\begin{equation} \label{eqn:14} \vcenter{\xymatrix{
(C,M_C) \ar[r]^<>(0.5)f \ar[d] & X^{\exp}_{\log} \ar[d] \\
(S,M_S) \ar[r] & V^{\exp}_{\log}
}} \end{equation}
in which 
\begin{enumerate}
\item $(C,M_C)$ is a logarithmically smooth curve over $(S,M_S)$,
\item $f$ maps the smooth locus of $C/S$ into the smooth locus of $X^{\exp} / V^{\exp}$,
\item the logarithmic structures are fine and saturated, and
\item the diagram is minimal.
\end{enumerate}
We write $\Kimstab(X/V)$ for the open substack of diagrams with finite automorphism groups.  We also write $\Kim(X/V)_{\log}$ for the stack with its minimal logarithmic structure.
\end{definition}

In the definition above, minimality refers to the following categorical condition~\cite{Gillam}:  \emph{Suppose that $M_S \rightarrow M'_S$ is a morphism of logarithmic structures on $S$ and let $M'_C$ be the logarithmic structure on $C$ obtained by pullback of $(C, M_C)$ via the morphism $(S, M'_S) \rightarrow (S, M_S)$.  Suppose further that the $(S, M'_S)$-point of $\Kim(X/V)_{\log}$ is induced from an $(S, M''_S)$-point, for a third logarithmic structure $M''_S$ and a morphism $M''_S \rightarrow M'_S$.  Then there exists a unique morphism of logarithmic structures $M_S \rightarrow M''_S$ making everything in sight compatible.}  

The following proposition shows that Definition~\ref{def:kim} is equivalent to Kim's definition \cite[Section~5.2.2]{Kim}.  In the statement and proof, we write $M_{C/S}$ for the logarithmic structure on $S$ associated to a family $C/S$ of pre-stable curves.

\begin{proposition} \label{prop:kim-min}
An $(S, M_S)$ object of $\Kim(X/V)_{\log}$ admits a unique morphism to a minimal such object.  A diagram~\eqref{eqn:14} is minimal if and only if
\begin{enumerate}[label=(\roman{*})]
\item for every geometric point $s$ of $S$, the rank of the cokernel of $g^\ast M_{V^{\exp}} \rightarrow M_S$ is equal to the number of nondistinguished nodes of $C_s$, 
\item \label{freesub} there is no locally free submonoid $N \subset \oM_S$ containing the image of $\oM_{C/S}$ except $\oM_S$ itself, and
\item the image of each irreducible element of $\oM_{C/S}$ in $\oM_S$ is a multiple of an irreducible element in $\oM_S$.
\end{enumerate}
\end{proposition}
\begin{proof}
Suppose diagram~\eqref{eqn:14} is an object of $\Kim(X/V)_{\log}$.  Then we have maps $M_{C/S} \rightarrow M_S$ and $g^\ast M_{V^{\exp}} \rightarrow M_S$ where $g$ denotes the structural map $S \rightarrow V^{\exp}$.  These must satisfy the following relation:  \emph{Let $\delta$ be the image in $M_S$ of the generator of $M_{C/S}$ corresponding to a node in a fiber of $C$ over $S$ that maps into the non-smooth locus of $X^{\exp}$ over $V^{\exp}$.  Let $\rho$ be the image of the corresponding generator of $M_{V^{\exp}}$.  Denote by $\odelta$ and $\orho$ the images of $\delta$ and $\rho$ in the characteristic monoid $\oM_S$.  Then $\odelta = c \orho$ for some positive integer $c$ (the contact order).}

Let $\oM^m_S$ be the fine saturated monoid generated by the elements $\odelta$ and $\orho$ corresponding to the non-smooth loci of $C$ and $X^{\exp}_S$ over $S$, with the relations $\odelta = c \orho$ as above.  Then there is a canonical map $\oM^m_S \rightarrow \oM_S$.  Define $M^m_S = \oM^m_S \fp_{\oM_S} M_S$.  Then the composition $M^m_S \rightarrow M_S \rightarrow \cO_S$ makes $M^m_S$ into a logarithmic structure on $S$.

Note that $\oM^m_S$ satisfies the properties of the proposition.  Moreover, if $\oM_S$ satisfied those properties already then $\oM^m_S \rightarrow \oM_S$ would be an isomorphism.

We check $M^m_S$ is minimal in the sense of Definition~\ref{def:kim}.  Suppose we have a map $M'_S \rightarrow M_S$ and an $(S, M'_S)$-point of $\Kim(X/V)_{\log}$ inducing~\eqref{eqn:14}.  Then we certainly obtain a map $\oM^m_S \rightarrow \oM'_S$ inducing the map $\oM^m_S \rightarrow \oM_S$ by composition with $\oM'_S \rightarrow \oM_S$ since the defining relations of $\oM^m_S$ must be satisfied in $\oM'_S$ as well as in $\oM_S$.  Since $M'_S = M_S \fp_{\oM_S} \oM'_S$ we obtain the map $M^m_S \rightarrow M'_S$ automatically.
\end{proof}

\begin{nlem}[\namer{\ref{lem:spaces}}{$\Kim$}{\ref{spaces:proper-DM}}]
The stack $\Kimstab(X/V)$ is a proper, Deligne--Mumford stack.
\end{nlem}
\begin{proof}
This follows from \cite[Main Theorem A]{Kim}.
\end{proof}

\begin{nlem}[\namer{\ref{lem:spaces}}{$\Kim$}{\ref{spaces:artin}}]
The stack $\Kim(\sX/\sV)$ is an Artin stack.
\end{nlem}
\begin{proof}
This follows again from \cite[Main Theorem A]{Kim}.
\end{proof}

\subsubsection{The relative obstruction theory}

\begin{nlem}[\namer{\ref{lem:spaces}}{$\Kim$}{\ref{spaces:cart}}]
The diagram \eqref{diag:naive} for $K=\Kim$ is cartesian.
\end{nlem}
\begin{proof}
This is immediate using the logarithmically commutative and cartesian diagrams
\begin{equation*} \vcenter{\xymatrix{
X^{\exp} \ar[r] \ar[d] & \sX^{\exp} \ar[d] \\
X \ar[r] & \sX
}} \qquad \text{and} \qquad \vcenter{\xymatrix{
V^{\exp} \ar[r] \ar[d] & \sV^{\exp} \ar[d] \\
V \ar[r] & \sV .
}} \end{equation*}
\end{proof}

As before, pulling back the relative obstruction theory $\sF$ for $\fM(X/V)$ over $\fM(\sX/\sV)  \fp_{\sV} V$ (section~\ref{sec:naive}) gives a relative obstruction theory $\sE'$ for $\Kim(X/V)$ over $\Kim(\sX/\sV) \fp_{\sV} V$.

\subsubsection{Kim's obstruction theory}

\notn{$\fMB$}{base for Kim's obstruction theory on $\Kim(X/V)$}

Kim constructed the virtual fundamental class on $\Kim(X/V)$ using an obstruction theory over the moduli space $\fMB$ (using Kim's notation \cite[Section~7.1]{Kim}) each of whose $S$-points consists of
\begin{enumerate}
\item a logarithmic structure $M_S$ on $S$,
\item a logarithmically smooth family of curves $(C, M_C)$ over $(S, M_S)$,
\item a logarithmic map $(S, M_S) \rightarrow \sV^{\exp}_{\log}$
\end{enumerate}
such that if we write $M'_S$ for the pullback of the logarithmic structure from $\sV^{\exp}$ then
\begin{enumerate}[resume]
\item the map $M'_S \rightarrow M_S$ is an \emph{extended simple map} \cite[Section~4.3]{Kim} of logarithmic structures, meaning that it satisfies the conditions of Proposition~\ref{prop:kim-min}.
\end{enumerate}
We will set $\sZ = \fMB \fp_{\sV} V$ in Lemma~\ref{lem:spaces}~\ref{spaces:obs}.

\begin{proposition}
$\fMB$ is smooth.
\end{proposition}

One proof may be found in \cite[sections~6.2.2 and~6.2.3]{Kim}; another is given below.

\begin{proof}
The logarithmic structure on $\fMB$ is locally free, so the map $\fMB \rightarrow \Log$ factors through the substack of locally free logarithmic structures.  This substack is smooth (Proposition~\ref{prop:loc-free}), so it is sufficient to show that $\fMB$ is smooth over $\Log$---i.e., that $\fMB$ is logarithmically smooth when it is given the logarithmic structure restricting to $M_S$ on an $S$-point.

Consider the stack $\fG$ whose $S$-points are extended simple maps $M'_S \rightarrow M_S$ of (locally free) logarithmic structures.  We have a projection $\fMB \rightarrow \fG \fp_{\Log} \Log(\fM)$ in which an $S$-point of the target is a scheme $S$ equipped with an extended simple map $M'_S \rightarrow M_S$ of log.\ structures and a family of log.\ smooth curves over $(S, M_S)$.  In fact, this projection is \'etale:  Recall from~\cite{ACFW} that $\sV^{\exp}$ is the moduli space of \emph{aligned} logarithmic structures.  Therefore the only additional data needed to lift a point of $\fG \fp_{\Log} \Log(\fM)$ is an alignment of the log.\ structure $M'_S$, and such alignments are parameterized by an scheme that is \'etale over $S$.

It will therefore be enough to show that $\fG \fp_{\Log} \Log(\fM)$ is smooth, and since $\Log(\fM)$ is smooth over $\Log$, it will even be enough to show that $\fG$ is smooth over $\Log$.  If $M'_S \rightarrow M_S$ is an extended simple map we can canonically identify $M_S$ with a product $M^{(1)}_S \times M^{(2)}_S$ where $M'_S \subset M^{(1)}_S$ is a simple extension and $M^{(2)}_S$ is free:  let $\oM^{(1)}_S$ be the submonoid consisting of those elements of $\oM_S$ which possess a multiple contained in $M'_S$ and let $M^{(1)}_S = \oM^{(1)}_S \fp_{\oM_S} M_S$; by definition of an extended simple map, $\oM^{(1)}_S$ is locally generated by a subset of the generators of the minimal set of generators of $\oM_S$; the complementary subset therefore generates a locally free sheaf of monoids $\oM^{(2)}_S$ and we take $M^{(2)}_S = \oM^{(2)}_S \fp_{\oM_S} M_S$.  Since the stack of locally free logarithmic structures is smooth (as we have remarked above), it is enough to show that the stack $\fG^0$ parameterizing simple extensions $M'_S \rightarrow M_S$ is smooth.

But the stack of pairs $(M', M'')$ where $M'$ is a free logarithmic structure and $M''$ is a simple extension of $M'$ is equivalent to the stack of pairs $(M'', \varphi)$ where $M''$ is a free logarithmic structure and $\varphi$ is a map from the generators of $\oM''$ to positive integers.  This is \'etale over the stack of locally free logarithmic structures, so it is smooth.
\end{proof}

The proposition implies that a virtual class for $\Kim(X/V)$ over $V^{\exp}$ can be defined by an obstruction theory relative to $\fMB \fp_{\sV} V$.  Kim's obstruction theory may be described as follows:  Consider a lifting problem
\begin{equation*} \xymatrix{
S \ar[r] \ar[d] & \Kim(X/V) \ar[d] \\
S' \ar@{-->}[ur] \ar[r] & \fMB
} \end{equation*}
corresponding to a \emph{logarithmic} lifting problem
\begin{equation*} \xymatrix{
(S, M_S) \ar[r] \ar[d] & \Kim(X/V) \ar[d] \\
(S',M_{S'}) \ar@{-->}[ur] \ar[r] & \fM_{\log} \fp V^{\exp}_{\log} .
} \end{equation*}
One translates this into the following logarithmic extension problem
\begin{equation*} \xymatrix{
(C, M_C) \ar[r] \ar[d] \ar@/^15pt/[rr]^<>(0.5)f & (C', M_{C'}) \ar@{-->}[r] \ar[d] & X^{\exp}_{\log} \ar[d]  \\
(S, M_S) \ar[r] & (S', M_{S'}) \ar[r] & V^{\exp}_{\log},
} \end{equation*}
which immediately simplifies to
\begin{equation*} \xymatrix{
(C,M_C) \ar[r]^<>(0.5)f \ar[d] & X^{\exp}_{\log} \ar[d]^{\varpi} \\
(C',M_{C'}) \ar@{-->}[ur] \ar[r] & V^{\exp}_{\log} .
} \end{equation*}
As $\varpi$ is logarithmically smooth, the lifts form a torsor under $\uHom(f^\ast \Omega_{\varpi}^{\log}, \pi^\ast J)$ on $C$ (where $J = I_{S/S'}$ is the ideal of $S$ in $S'$).  Therefore, if we define $\sE(S,J)$ to be the category of torsors under $\uHom(f^\ast \Omega_{\varpi}^{\log}, \pi^\ast J)$ we get a relative obstruction theory for $\Kim(X/V)$ over $\fMB \fp_{\sV} V$.

\begin{lemma} \label{lem:kim-perf}
$\sE$ is a perfect relative obstruction theory.
\end{lemma}
\begin{proof}
Since $f^\ast \Omega^{\log}_{\varpi}$ is a vector bundle, we can identify 
\begin{equation*}
\RR \Hom(f^\ast \Omega^{\log}_{\varpi}, \pi^\ast J) = \RR \pi_\ast \bigl( f^\ast T^{\log}_{\varpi} \tensor \pi^\ast J \bigr) = \RR \pi_\ast \bigl( f^\ast T^{\log}_{\varpi} \bigr) \tensor J .
\end{equation*}
Thus, $\sE(S,J)$ is representable by $R \pi_\ast \bigl( f^\ast T^{\log}_{\varpi} \bigr)^\vee [1]$,
\Jonathan{added shift}
which is perfect in cohomological degrees $[-1,0]$ because $f^\ast T^{\log}_{\varpi}$ is a vector bundle and $C$ has cohomological dimension~$1$.
\end{proof}

\begin{nlem}[\namer{\ref{lem:spaces}}{$\Kim$}{\ref{spaces:unobs}}] \label{lem:kim-unobs}
The projection $\Kim(\sX/\sV) \rightarrow \fMB$ is \'etale and unobstructed.
\end{nlem}
\begin{proof}
Since the map is strict, it's the same to show it is logarithmically \'etale.  Consider a logarithmic lifting problem
\begin{equation*} \xymatrix{
(S,M_S) \ar[r] \ar[d] & \Kim(\sX/\sV) \ar[d] \\
(S',M_{S'}) \ar[r] \ar@{-->}[ur] & \fMB .
} \end{equation*}
This translates into a logarithmic extension problem
\vskip1pt
\begin{equation*} \xymatrix{
(C,M_C) \ar[r] \ar@/^15pt/[rr] \ar[d] & (C',M_{C'}) \ar[d] \ar@{-->}[r] & \sX^{\exp}_{\log} \ar[d] \\
(S,M_S) \ar[r] & (S',M_{S'}) \ar[r] & \sV^{\exp}_{\log} ,
} \end{equation*}
which we argue has a unique solution.  Indeed, $\sX^{\exp}_{\log} \rightarrow \sV^{\exp}_{\log}$ is logarithmically \'etale by Corollary~\ref{cor:log-stacks}.
\end{proof}

\begin{nlem}[\namer{\ref{lem:spaces}}{$\Kim$}{\ref{spaces:obs}}]
There is a compatible sequence of obstruction theories for the sequence of maps
\begin{equation*} 
\Kim(X/V) \rightarrow \Kim(\sX/\sV) \fp_{\sV} V \rightarrow \fMB \fp_{\sV} V
\end{equation*}
such that the relative obstruction theory for $\Kim(\sX/\sV)$ over $\fMB$ is the canonical one (whose relative virtual fundamental class is the fundamental class of $\Kim(\sX/\sV)$).
In the notation of \hyperref[spaces:obs]{Lemma~\ref*{lem:spaces}~\ref*{spaces:obs}}, $\sZ = \fMB \fp_{\sV} V$.
\end{nlem}
\begin{proof}
In fact, the relative obstruction theory for $\Kim(X/V)$ over $\Kim(\sX/\sV)$ is \emph{identical} to the one over $\fMB$, since 
\begin{equation*}
\Omega_{X^{\exp} / V^{\exp}}^{\log} = p^\ast \Omega_{X/V}^{\log}
\end{equation*}
in virtue of the fact that $X^{\exp} \rightarrow X \fp_V V^{\exp}$ is logarithmically \'etale.  As the map $\Kim(\sX/\sV) \rightarrow \fMB$ is \'etale with a trivial relative obstruction theory, this means that the obstruction theories are compatible.
\end{proof}

\begin{nlem}[\namer{\ref{lem:spaces}}{$\Kim$}{\ref{spaces:dense}}] \label{cor:kim-dense}
The locus of totally nondegenerate maps is dense in $\Kim(\sX/\sV)$.
\end{nlem}
\begin{proof}
This open substack can be identified as the substack where the natural logarithmic structure is trivial.  But by \hyperref[lem:kim-unobs]{Lemma~\namer{\ref*{lem:spaces}}{$\Kim$}{\ref*{spaces:unobs}}}, the stack $\Kim(\sX/\sV)$ is \'etale and strict over $\fMB$, which is logarithmically smooth.  Hence $\Kim(\sX/\sV)$ is logarithmically smooth, so the locus where its logarithmic structure is trivial is a dense open substack.
\end{proof}

\subsection{Unexpanded logarithmic targets:  the theory of Gross--Siebert and  Abramovich--Chen} \label{sec:acgs}

As in the last section, we give $\sX$ and $\sV$ their natural logarithmic structures.  In fact, $\sX$ and $\sV$ may each be interpreted as moduli spaces of logarithmic structures of certain types:  see \cite[Section~8]{ACFW}.  The spaces $X$ and $V$ possess logarithmic structures pulled back
\Jonathan{typo corrected}
via the maps $X \rightarrow \sX$ and $V \rightarrow \sV$.

Abramovich and Chen \cite{AC,Chen} and Gross and Siebert \cite{GS} have defined a moduli space of logarithmic stable maps from logarithmically smooth curves into a logarithmic target.

\begin{definition} \notn{$\ACGS(X/V)$}{maps from log.\ smooth curves to $X/V$}
A \emph{stable logarithmic map} \cite[Definiton~1.5]{GS} into $X/V$ is a logarithmically commutative diagram
\begin{equation*} \xymatrix{
C \ar[r] \ar[d] & X \ar[d] \\
S \ar[r] & V
} \end{equation*}
in which $C$ is a pre-stable logarithmically smooth curve over $S$ \cite[Definition~1.3]{GS}.  Such an object is called \emph{basic} \cite[Definition~1.19]{GS} (or \textit{minimal}
\Jonathan{italics added}
\cite[Definition~3.5.1]{Chen}) if its fibers over $S$ are basic.  The substack of basic logarithmic maps to $X/V$ will be denoted $\ACGS(X/V)$.  An object of $\ACGS(X/V)$ is called stable if its automorphism group is finite (\cite[Definition~1.3]{GS} and \cite[Definition~3.6.1]{Chen}).   We will write $\ACGSstab(X/V)$ for the substack of stable objects (\cite[Defintion~2.1]{GS} and \cite[Defintion~3.6.5]{Chen}).  We decorate the notation for these stacks with the subscript $\log$ to indicate the corresponding stack on the category of logarithmic schemes.
\end{definition}

\begin{nlem}[\namer{\ref{lem:spaces}}{$\ACGS$}{\ref{spaces:proper-DM}}]
The stack $\ACGSstab(X/V)$ is a proper, Deligne--Mumford stack.
\end{nlem}
\begin{proof}
See \cite[Corollary~2.8 and Corollary~4.2]{GS}, \cite[Theorem~3.6.6 and Proposition~3.8.1]{Chen}, and \cite[Theorem~5.8]{AC}.
\end{proof}

\begin{nlem}[\namer{\ref{lem:spaces}}{$\ACGS$}{\ref{spaces:artin}}]
The stack $\ACGS(\sX/\sV)$ is an Artin stack.
\end{nlem}
\begin{proof}
See \cite[Corollary~2.6]{GS} and \cite[Corollary~3.5.4]{Chen}.
\end{proof}

\begin{nlem}[\namer{\ref{lem:spaces}}{$\ACGS$}{\ref{spaces:unobs}}] \label{lem:acgs-etale}
The projection $\ACGS(\sX / \sV)_{\log} \rightarrow \fM_{\log} \fp \sV_{\log}$ is logarithmically \'etale and unobstructed.
\end{nlem}
\begin{proof}
The (logarithmic) lifting problem here is
\begin{equation*} \xymatrix{
C \ar[r] \ar[d]_\pi \ar@/^15pt/[rr] & C' \ar[d] \ar@{-->}[r] & \sX \ar[d] \\
S \ar[r] & S' \ar[r] & \sV .
} \end{equation*}
The existence and uniqueness of the dashed arrow are immediate from the fact that $\sX$ is logarithmically \'etale over $\sV$ (Corollary~\ref{cor:log-stacks}).
\end{proof}

\begin{nlem}[\namer{\ref{lem:spaces}}{$\ACGS$}{\ref{spaces:dense}}] \label{cor:acgs-dense} 
The locus of totally nondegenerate maps is dense in $\ACGS(\sX/\sV)$.
\end{nlem}
\begin{proof}
This locus is the same as the open substack of $\ACGS(\sX/\sV)$ where the logarithmic structure is trivial, which is dense because $\ACGS(\sX/\sV)$ is logarithmically smooth.  We thank Q.\ Chen for pointing out this simple proof.
\end{proof}

\subsubsection{The relative obstruction theory}

\begin{nlem}[\namer{\ref{lem:spaces}}{$\ACGS$}{\ref{spaces:cart}}]
The diagram \eqref{diag:naive} for $K=\ACGS$ is cartesian.
\end{nlem}
\begin{proof}
Immediate from the fact that $X \rightarrow V \fp_{\sV} \sX$ is strict.
\end{proof}

As usual, this implies that we get a relative obstruction theory $\sE'$ for $\ACGS(X/V)$ over $\ACGS(\sX/\sV) \fp_{\sV} V$ by pulling back the obstruction theory $\sF$ of $\fM(X/V)$ over $\fM(\sX/\sV) \fp_{\sV} V$ (section~\ref{sec:naive}).

\subsubsection{The obstruction theory of Gross and Siebert}

Now we recall the definition of the obstruction theory from \cite[Section~5]{GS}.

For any $X/V$ we have a logarithmic map $\ACGS(X/V)_{\log} \rightarrow \fM_{\log} \fp V_{\log}$.  This corresponds to a map of stacks over schemes $\ACGS(X/V) \rightarrow \Log(\fM) \fp_{\Log} \Log(V)$, where the $S$-points of the target
\Jonathan{changed a word}
consist of a logarithmic structure $M_S$ on $S$ and a pair of logarithmic maps $(S, M_S) \rightarrow \fM$ and $(S, M_S) \rightarrow V$.

The obstruction theory for $\ACGS(X/V)$ is defined relative to $\sZ := \Log(\fM) \fp_{\Log} \Log(V)$.  Indeed, a lifting problem
\begin{equation*} \xymatrix{
S \ar[r] \ar[d] & \ACGS(X/V) \ar[d] \\
S' \ar@{-->}[ur] \ar[r] & \Log(\fM) \fp_{\Log} \Log(V)
} \end{equation*}
corresponds to a logarithmic lifting problem
\begin{equation*} \xymatrix{
(C,M_C) \ar[r] \ar[d]_\pi \ar@/^15pt/[rr]^f & (C',M_{C'}) \ar[d] \ar@{-->}[r] & X \ar[d] \\
(S,M_S) \ar[r] & (S',M_{S'}) \ar[r] & V .
} \end{equation*}
Solutions to this problem naturally form a torsor on $C$ under $f^\ast T_{X/V}^{\log} \tensor \pi^\ast J$, where $J = I_{S/S'}$ is the ideal of $S$ in $S'$.  Therefore if we define $\sE(S,J)$ to be the category of torsors on $C$ under $f^\ast T_{X/V}^{\log} \tensor \pi^\ast J$, we obtain a relative obstruction theory $\sE$ for $\ACGS(X/V)$ over $\Log(\fM) \fp_{\Log} \Log(V)$.

\begin{lemma}
$\sE$ is a perfect relative obstruction theory.
\end{lemma}
\begin{proof}
We have already seen that it is an obstruction theory.  The proof of perfection is identical to the proof of Lemma~\ref{lem:kim-perf}.
\end{proof}
%The relative obstruction theory for $\ACGS(X/V)$ over $\ACGS(\sX/\sV)$ comes from the (non-logarithmic) lifting problem
%\begin{equation*} \xymatrix{
%& & X \ar[d] \\
%C \ar[r] \ar@/^10pt/[urr]^f \ar[d]_\pi & C' \ar[d] \ar@{-->}[ur] \ar[r] & \sX \fp_{\sV} V \ar[d] \\
%S \ar[r] & S' \ar[r] & V .
%} \end{equation*}
%These lifts correspond to lifts of
%\begin{equation*} \xymatrix{
%C \ar[r] \ar[d] & X \ar[d] \\
%C' \ar@{-->}[ur] \ar[r] & \sX \fp_{\sV} V ,
%} \end{equation*}
%which form a torsor on $C$ under $f^\ast T_{X / \sX \fp_{\sV} V} \tensor \pi^\ast J$ since $X$ is smooth over $\sX \fp_{\sV} V$ by hypothesis.  Therefore taking $\sF(S,J)$ to be the category of torsors under $f^\ast T_{X/ \sX \fp_{\sV} V} \tensor \pi^\ast J$ gives a relative obstruction theory over $\ACGS(\sX/\sV)$.

\subsubsection{Comparison of the obstruction theories}

\begin{nlem}[\namer{\ref{lem:spaces}}{$\ACGS$}{\ref{spaces:obs}}]
There is a compatible sequence of obstruction theories for the sequence of maps
\begin{equation*}
\ACGS(X/V) \rightarrow \ACGS(\sX/\sV) \fp_{\sV} V \rightarrow \Log(\fM) \fp_{\Log} \Log(V)
\end{equation*}
such that the relative obstruction theory for $\ACGS(\sX/\sV)$ over $\fM \fp_{\Log} \sV$ is the canonical one (whose relative virtual fundamental class is the fundamental class).  In the language of \hyperref[spaces:obs]{Lemma~\ref*{lem:spaces}~\ref*{spaces:obs}}, $\sZ = \Log(\fM) \fp_{\Log} \Log(V)$.
\Jonathan{added last sentence}
\end{nlem}
\begin{proof}
Both $\sE(S,J)$ and $\sF(S,J)$ have been defined as the category of torsors under the sheaf of groups $f^\ast T_{X/V}^{\log} \tensor \pi^\ast J$.  The relative obstruction theories of $\ACGS(X/V)$ over $\ACGS(\sX/\sV) \fp_{\sV} V$ and over $\Log(\fM) \fp_{\Log} \Log(V)$ are therefore the same.  By \hyperref[lem:acgs-etale]{Lemma~\namer{\ref*{lem:spaces}}{$\ACGS$}{\ref*{spaces:unobs}}}, the projection $\ACGS(\sX/\sV) \fp_{\sV} V \rightarrow \Log(\fM) \fp_{\Log} \Log(V)$ is \'etale, so we may choose its relative obstruction theory to be trivial.\footnote{In fact, we don't have to \emph{choose} the relative obstruction theory for $\ACGS(\sX/\sV)$ over $\Log(\fM) \fp_{\Log} \Log(V)$ to be trivial:  if we were to apply the definition of the relative obstruction theory for $\ACGS(X/V)$ over $\Log(\fM) \fp_{\Log} \Log(V)$ to the case $X = \sX$ and $V = \sV$, we would discover that it is trivial because $\sX$ is logarithmically \'etale over $\sV$.}  We therefore obtain a compatible sequence of obstruction theories.
\end{proof}

\section{Maps between moduli spaces} \label{sec:maps}
\Jonathan{changed section title}

In this section, we describe the maps $\Psi, \Theta,$ and $\Upsilon$, along with their corresponding cartesian diagrams~\eqref{eqn:4}.  In each case we prove the remaining parts of Lemma~\ref{lem:maps}, thus completing the proof of Theorem~\ref{thm:compare}.

\subsection{Orbifold expanded stable maps to relative stable maps:  the cartesian square for $\Psi$}

Since $\Psi$ will not collapse any components in the source curves there is no need to modify $\AF(\sX/\sV)$ and diagram~\eqref{eqn:4} takes the form
 \begin{equation}\label{diag:psi} \vcenter{\xymatrix{
      \AFstab(X/V) \ar[r]^\Psi \ar[d] & \JListab(X/V) \ar[d] \\
      \AF(\sX/\sV) \ar[r] & \JLi(\sX/\sV) .
    }}
\end{equation}
We construct this diagram as follows.  To an orbifold expanded stable map
\begin{equation*} \xymatrix{
\tC \ar[r] \ar[d] & X^{\bfr} \ar[d] \\
S \ar[r] & V^{\bfr}
} \end{equation*}
the map $\Psi$ assigns a relative stable map
\begin{equation*} \xymatrix{
C\ar[r]\ar[d]& X^{\exp}\ar[d]\\
S\ar[r]& V^{\exp}\\
} \end{equation*}
in which $C$ is the relative coarse moduli space of the map $\tC \rightarrow X^{\exp} \fp_{V^{\exp}} S$ induced from the composition with the untwisting morphisms (see \cite[Section~2]{AF} or \cite[Section~7]{ACFW}):
\begin{equation*} \xymatrix{
\tC\ar[r] \ar[d]& X^{\bfr}\ar[r]\ar[d]& X^{\exp}\ar[d]\\
S\ar[r] & V^{\bfr}\ar[r]& V^{\exp}.
} \end{equation*}
Applying this with $X = \sX$ and $V = \sV$ gives the construction of $\AF(\sX/\sV) \rightarrow \JLi(\sX/\sV)$.
%Note that the map $\overline{C}\to X^{\exp}$ factors through $S\fp_{V^{\exp}}X^{\exp}$, so taking the coarse moduli space of $C$ relative to $S\fp_{V^{\exp}}X^{\exp}$ will produce the same result, ensuring that our map is well defined.  The construction of the map $\AF(\sX/\sV) \to \JLi(\sX/\sV)$ is analogous.

The left and right vertical arrows in diagram~\eqref{diag:psi} are given by the compositions
\begin{equation*} \vcenter{\xymatrix{
\tC \ar[r] \ar[d] & X^{\bfr} \ar[d] \ar[r]& \sX^{\bfr}\ar[d]\\
S \ar[r] & V^{\bfr}\ar[r] & \sV^{\bfr}
}} 
\qquad  \text{and} \qquad
\vcenter{\xymatrix{
\tC \ar[r] \ar[d] & X^{\exp} \ar[d] \ar[r]& \sX^{\exp}\ar[d]\\
S \ar[r] & V^{\exp}\ar[r] & \sV^{\exp}
}} \end{equation*}
respectively.  This diagram is commutative by construction.  In order to show it is cartesian, we begin with the following lemma.

\begin{lemma}\label{lem:aftoli}
The diagram
\begin{equation*} \xymatrix{
      \AF(X/V) \ar[r] \ar[d] & \JLi(X/V) \ar[d] \\
      \AF(\sX/\sV) \ar[r] & \JLi(\sX/\sV) .
    }
\end{equation*}
is cartesian.
\end{lemma}
\begin{proof}
The maps in the diagram are given by the same constructions as above.  The lemma is an immediate application of the cartesian diagrams
\begin{equation*} \vcenter{\xymatrix{
X^{\bfr} \ar[r] \ar[d] & X^{\exp} \ar[d] \\
\sX^{\bfr} \ar[r] & \sX^{\exp}
}} \qquad \text{and} \qquad \vcenter{\xymatrix{
V^{\bfr} \ar[r] \ar[d] & V^{\exp} \ar[d] \\
\sV^{\bfr} \ar[r] & \sV^{\exp}.
}} \end{equation*}
\end{proof}

\begin{nlem}[\namer{\ref{lem:maps}}{$\Psi$}{\ref{maps:DM}}]\label{prop:DM:AFtoLi}
The map $\AF(\sX/\sV) \rightarrow \JLi(\sX/\sV)$ is of relative Deligne--Mumford type.
\end{nlem}
\begin{proof}
To see the map is of Deligne--Mumford type,  we must show that, \emph{given an $S$-point $\gamma$ of $\AF(\sX/\sV)$, corresponding to the diagram 
\begin{equation*} \xymatrix{
\tC\ar[r]\ar[d]& \sX^{\bfr}\ar[d]\\
S\ar[r]& \sV^{\bfr},
} \end{equation*}
with image $\alpha \in\JLi(\sX/\sV)(S)$, the group $G$ of automorphisms of $\gamma$ inducing the identity on $\alpha$ is finite.}  

We have a natural homomorphism $G \to \Aut_C(\tC)$ by taking the induced automorphism on the source twisted curves, and $\Aut_C(\tC)$ is finite (see \cite[Proposition~7.1.1]{ACV}).  We may consider an exact sequence 
\[1\to K \to G \to \text{Aut}_C(\tC)\]
where the kernel $K$ consists of automorphisms of $\gamma$ inducing the identity on $\alpha$ and furthermore inducing the identity automorphism of the twisted curve $\tC$ over $C$.  We are now left to show that the group $K$ of automorphisms
\[\Biggl\{ 
\UseTwocells
\xymatrix@C+2pc{
\tC \rtwocell^{f}_{f}{\;\;\;} & X^{\bfr}_S \ar[r]& X^{\exp}_S
}
\Biggr\}
\]
is finite.  Let $W = X^{\bfr}_S \fp_{X^{\exp}_S} X^{\bfr}_S$.\Jonathan{changed this; please check it's correct}
%Let $W = X^{\bfr} \fp_{X^{\exp} \fp_{V^{\exp}} V^{\bfr}} X^{\bfr}$.  

Our task is equivalent to showing that the set of lifts of the diagram
\[
\xymatrix{
&&X^{\bfr}\ar[d]^{\Delta}\\
\tC\ar@{-->}[urr] \ar[rr]_<>(0.5){(f,f)}&& W
}
\]
is finite.  Here $\Delta$ is the diagonal map, which is finite and unramified because $X^{\bfr}_S \rightarrow X^{\exp}_S$ is of Deligne--Mumford type.  Such a lift corresponds to a section of $Z := X^{\bfr}_S \fp_W \tC$ over $\tC$.  But sections of a separated, unramified map are open and closed; since $\tC$ is connected, this means that the number of sections of $Z$ over $\tC$ is bounded by, for example, the number of connected components of $Z$, which is finite because $Z$ is finite over $\tC$. 
\end{proof}

\begin{nlem}[\namer{\ref{lem:maps}}{$\Psi$}{\ref{maps:cart}}]
Diagram~\eqref{diag:psi} is cartesian.
\end{nlem}
\begin{proof}
Note first that stability in both $\AF(X/V)$ and $\JLi(X/V)$ may be characterized in terms of the same conditions concerning special points on rational components of the source curve.  Therefore the image of a point of $\AFstab(X/V)$ under the map $\AF(X/V) \rightarrow \JLi(X/V)$ lies in $\JListab(X/V)$.  It remains to demonstrate the converse:  \textit{if $\alpha$ is an $S$-point of $\AF(X/V)$ whose image $\beta$ in $\JLi(X/V)$ is stable then $\alpha$ is stable.}

We have an exact sequence
\Jonathan{adjusted formatting slightly}
\begin{equation*}
1 \rightarrow \Aut_{\AF(X/V) \big/ \JLi(X/V)}(\alpha) \rightarrow \Aut_{\AF(X/V)}(\alpha) \rightarrow \Aut_{\JLi(X/V)}(\beta).
\end{equation*}
By hypothesis $\Aut_{\JLi(X/V)}(\beta)$ is finite and $\Aut_{\AF(X/V) \big/ \JLi(X/V)}(\beta)$ is finite by \hyperref[prop:DM:AFtoLi]{Lemma~\namer{\ref*{lem:maps}}{$\Psi$}{\ref*{maps:DM}}} so we may conclude that $\Aut_{\AF(X/V)}(\alpha)$ is finite as well.
\Jonathan{changed proof to avoid deleted lemma}
\end{proof}

\subsection{Expanded targets and expanded orbifold targets:
The cartesian square for $\Theta$}

Once again, $\Theta$  will not collapse any components in the source curves. Thus there is no need to modify $\Kim(\sX/\sV)$ and diagram~\eqref{eqn:4} takes the form
\begin{equation}\label{diag:theta} \vcenter{\xymatrix{
\Kimstab(X/V) \ar[r]^<>(0.5)\Theta \ar[d] & \JListab(X/V) \ar[d] \\
\Kim(\sX/\sV) \ar[r] & \JLi(\sX/\sV) .
}} 
\end{equation}
The arrow $\Theta$ is given by sending a diagram of logarithmic schemes
\begin{equation*} \xymatrix{
C \ar[r] \ar[d] & X^{\exp} \ar[d] \\
S \ar[r] & V^{\exp}
} \end{equation*}
to the corresponding diagram of underlying schemes, forgetting the logarithmic structures.  The bottom arrow likewise forgets logarithmic structures.  The vertical arrows in diagram~\eqref{diag:theta} are given by compositions in
\begin{equation*} \xymatrix{
C \ar[r] \ar[d] & X^{\exp} \ar[d] \ar[r]& \sX^{\exp}\ar[d]\\
S \ar[r] & V^{\exp}\ar[r] & \sV^{\exp} ,
} \end{equation*}
viewed appropriately as a diagram in the category of logarithmic schemes or of schemes.

\begin{nlem}[\namer{\ref{lem:maps}}{$\Theta$}{\ref{maps:cart}}]
Diagram~\eqref{diag:theta} is cartesian.
\end{nlem}
\begin{proof}  This is immediate from the fact that the diagrams
\begin{equation*} \vcenter{\xymatrix{
X^{\exp} \ar[r] \ar[d] & \sX^{\exp} \ar[d] \\
X \ar[r] & \sX
}} \qquad \text{and} \qquad \vcenter{\xymatrix{
V^{\exp} \ar[r] \ar[d] & \sV^{\exp} \ar[d] \\
V \ar[r] & \sV ,
}} \end{equation*}
are logarithmically cartesian with strict horizontal arrows.
\end{proof}

\begin{lemma} \label{lem:kim-m-dm}
The projection $\Kim(X/V) \rightarrow \fM$ is of Deligne--Mumford type.
\end{lemma}
\begin{proof}
We must show that for a geometric point~\eqref{eqn:14} of $\Kim(X/V)$ (with $S$ the spectrum of a separably closed field), the group of automorphisms fixing $C$ is finite.  Such automorphisms come entirely from automorphisms of the logarithmic structure $M_S$ on $S$ that respect the map $M_{C/S} \rightarrow M_S$ (here $M_{C/S}$ is the log structure on $S$ canonically associated to the family of nodal curves $C/S$).  By Proposition~\ref{prop:kim-min}~\ref{freesub}, if $\overline{e}$ is a generator of $\oM_S$, some multiple $ke$ of $e$ lies in the image of $\oM_{C/S}$.  Therefore if $e$ is a lift of $\overline{e}$ to $M_S$, the only possible images of $e$ under automorphisms $M_S \rightarrow M_S$ fixing the image of $\oM_{C/S}$ send $e$ to $\zeta e$ where $\zeta$ is a $k$-th root of unity.  As $\oM_S$ has only finitely many generators, it follows that the automorphism group of $M_S$ fixing $M_{C/S}$ is finite.
\end{proof}

\begin{nlem}[\namer{\ref{lem:maps}}{$\Theta$}{\ref{maps:DM}}] \label{lem:kim-li-dm}
The map $\Kim(\sX/\sV) \rightarrow \JLi(\sX/\sV)$ is of Deligne--Mumford type.
\Jonathan{changed ``representable'' to ``of Deligne--Mumford type''}
\end{nlem}
\begin{proof}
By Lemma~\ref{lem:kim-m-dm}, $\Kim(\sX/\sV)$ is of Deligne--Mumford type over $\fM$ so it is also of Deligne--Mumford type over $\JLi(\sX/\sV)$.
\end{proof}

\begin{lemma}\label{lem:kim-li-surj}
The map $\Kim(\sX/\sV) \rightarrow \JLi(\sX/\sV)$ is surjective.
\end{lemma}
\begin{proof}
One must check that if $S$ is the spectrum of an algebraically closed field, an $S$-point of $\JLi(\sX/\sV)$ can be lifted to an $S$-point of $\Kim(\sX/\sV)$.  In \cite[Section~1.1 and Proposition~1.8]{Li2}, Li constructs canonical log structures on the expanded target, source curve, and base scheme of a predeformable map over $S$.  These constructions do not guarantee that the resulting logarithmic structures are saturated.  However, replacing them with their saturations yields a point of $\Kim(\sX/\sV)$.
\end{proof}

\subsection{Expanded logarithmic stable maps to logarithmic stable maps:
the cartesian square for $\Upsilon$} 
The construction of diagram~\eqref{eqn:4} is more involved this time.  We define stacks $\KimACGS(X/V)$, $\KimACGS(\sX/\sV)$ and $\KimACGSstab(\sX/\sV)$, the last of which will play the role of $\Kim^\ast(\sX/\sV)$. 
 
\begin{definition} \label{def:kimacgs}
\notn{$\KimACGS(X/V)$}{modification of $\Kim(X/V)$ with contracted curve}
We denote by $\KimACGS(X/V)$ the stack of logarithmically commutative  diagrams
\begin{equation} \label{eqn:8} \vcenter{\xymatrix{
C \ar[r] \ar[d] & X^{\exp} \ar[d] \ar[dr] \\
\oC \ar[r] \ar[d]  & X \fp_{V} V^{\exp} \ar[r] \ar[d] & X \ar[d] \\
S \ar[r] & V^{\exp} \ar[r] & V 
}} \end{equation}
in which
\begin{enumerate}[label=(\roman{*})]
\item $C$ and $\oC$ are logarithmically smooth curves over $S$;
\item the stabilization $C' \rightarrow \oC$ of the map $C \rightarrow \oC$ is an isomorphism;
\item the logarithmic structure on $S$ is minimal.
\end{enumerate}
We write $\KimACGSstab(X/V) \subset \KimACGS(X/V)$ for the substack given by imposing the following relative stability condition:
\begin{enumerate}[resume, label=(\roman{*})]
\item \label{enum:stability} the map $C \rightarrow X^{\exp}_S \fp_{X_S} \oC$ has finite automorphism group.
\end{enumerate}
\end{definition} 
\noindent Here minimality of the logarithmic structure on $S$ means simply that the induced diagram
\begin{equation*} \xymatrix{
C \ar[r] \ar[d] & X^{\exp} \ar[d] \\
S \ar[r] & V^{\exp}
} \end{equation*}
is a point of $\Kim(X/V)$.  This gives a map $\KimACGS(X/V) \rightarrow \Kim(X/V)$ by forgetting all of diagram~\eqref{eqn:8} except the square above.

We wish to construct a cartesian diagram
\begin{equation} \label{eqn:7} \vcenter{\xymatrix{
\Kimstab(X/V) \ar[r]^<>(0.5)\Upsilon \ar[d] & \ACGSstab(X/V) \ar[d] \\
\KimACGSstab(\sX/\sV) \ar[r]^<>(0.5)\tUpsilon & \ACGS(\sX/\sV) .
}} \end{equation}
We begin by defining $\tUpsilon : \KimACGS(X/V) \rightarrow \ACGS(X/V)$ to be the map taking a diagram~\eqref{eqn:8} (with $X$ and $V$ replaced by $\sX$ and $\sV$) 
\Jonathan{added parenthetical}
to the square
\begin{equation*} \xymatrix{
\oC \ar[r] \ar[d] & \sX \ar[d] \\
S \ar[r] & \sV .
} \end{equation*}
The lower horizontal arrow of diagram~\eqref{eqn:7} is the restriction of $\tUpsilon$ to $\KimACGSstab(\sX/\sV)$.  The arrow $\Upsilon$ is described in \cite[Proposition~6.3]{GS}; we give an alternate construction using
\Jonathan{``in'' changed to ``using''}
Appendix~\ref{app:chains}.  It is given by the composition
\begin{equation*} \xymatrix{
C \ar[r] \ar[d] & X^{\exp} \ar[d] \ar[r] & X\ar[d] \\
S \ar[r] & V^{\exp}\ar[r] & V,
} \end{equation*}
followed by stabilization of the outside rectangle to $C \xrightarrow{\tau} \oC \rightarrow X$ (where $\oC$ has the log.\ structure $\tau_\ast M_C$:  see Theorem~\ref{thm:log-pf}), and finally by \emph{replacing the logarithmic structures on $S$ and $\oC$ with those that are minimal in the sense of \cite{Chen} (or, equivalently, basic in the sense of \cite{GS}}).  The left and right vertical arrows of diagram~\eqref{eqn:7} are given by the compositions
\begin{equation*} \vcenter{\xymatrix{
C\ar[r]\ar[d]&X^{\exp} \ar[r] \ar[d] & \sX^{\exp} \ar[d]\\%\ar[r]& \sX\ar[d] \\
S\ar[r]&V^{\exp} \ar[r] & \sV^{\exp}%\ar[r]&\sV
}} \qquad \text{and} \qquad \vcenter{\xymatrix{
C\ar[r]\ar[d]& X \ar[r] \ar[d] & \sX \ar[d] \\
S\ar[r]&V \ar[r] & \sV
}} \end{equation*}
respectively, and for the left vertical arrow, $\oC$ is constructed as the relative stabilization of the map $C \rightarrow X_S = X \fp_V S$.

\begin{remark}
One cannot simply take $\Kim(\sX/\sV)$ for $\Kim^\ast(\sX/\sV)$ because $\Kimstab(X/V) \rightarrow \ACGSstab(X/V)$ may collapse components of the source curve in order to stabilize, while $\Kim(\sX/\sV) \rightarrow \ACGS(\sX/\sV)$ will not collapse anything. 
\end{remark}

\begin{remark}
There are two natural maps $\KimACGS(\sX/\sV) \rightarrow \ACGS(\sX/\sV)$.  In addition to $\tUpsilon$, a second map $\tUpsilon'$ takes diagram~\eqref{eqn:8} to the square
\begin{equation*} \xymatrix{
C \ar[r] \ar[d] & \sX \ar[d] \\
S \ar[r] & \sV .
} \end{equation*}
This second map fits into a commutative triangle
\begin{equation*} \xymatrix{
\KimACGS(\sX/\sV) \ar[d] \ar[dr]^{\tUpsilon'} \\
\Kim(\sX/\sV) \ar[r] & \ACGS(\sX/\sV) 
} \end{equation*}
but we will have no use for it. The analogous triangle with $\tUpsilon'$ replaced by $\tUpsilon$ is \emph{not} commutative.
\end{remark}
% This is just a diagram of stacks, not of logarithmic stacks.  The logarithmic structures intervene as follows:  the large rectangle on the left, together with a logarithmic structure $M$ on $S$, is an $S$-point of $\Kim(\sX/\sV)$.  The large rectangle on the bottom, together with a logarithmic structure $M'$ on $S$, is an $S$-point of $\ACGS(\sX/\sV)$.  There is also specified a morphism of logarithmic structures $M' \rightarrow M$ making the whole thing compatible. \Jonathan{make this precise} 

\begin{lemma}\label{lem:kimtoacgs}
The diagrams
\begin{equation*} \vcenter{\xymatrix{
\KimACGS(X/V) \ar[r]^<>(0.5){\tUpsilon} \ar[d] & \ACGS(X/V) \ar[d] \\
\KimACGS(\sX/\sV) \ar[r]^<>(0.5){\tUpsilon} & \ACGS(\sX/\sV)
}} \quad \text{and} \quad \vcenter{\xymatrix{
\KimACGSstab(X/V) \ar[r]^<>(0.5){\tUpsilon} \ar[d] & \ACGS(X/V) \ar[d] \\
\KimACGSstab(\sX/\sV) \ar[r]^<>(0.5){\tUpsilon} & \ACGS(\sX/\sV)
}} \end{equation*}
are cartesian.
\end{lemma}
\begin{proof}
The logarithmically cartesian diagrams
\begin{equation*} \vcenter{\xymatrix{
X^{\exp} \ar[r] \ar[d] & \sX^{\exp} \ar[d] \\
X \ar[r] & \sX
}} \qquad \text{and} \qquad \vcenter{\xymatrix{
V^{\exp} \ar[r] \ar[d] & \sV^{\exp} \ar[d] \\
V \ar[r] & \sV 
}} \end{equation*}
immediately imply that the first diagram is cartesian.  This also implies that $X_S^{\exp} \fp_{X_S} \oC \to \sX_S^{\exp} \fp_{\sX_S} \oC$ is an isomorphism. It follows that the map $C \rightarrow X_S^{\exp} \fp_{X_S} \oC$ is stable if and only if $C \rightarrow \sX_S^{\exp} \fp_{\sX_S} \oC$ is stable, showing that the second diagram in the lemma is cartesian.  
\end{proof}

Unfortunately the first diagram of the lemma does not remain cartesian when $\ACGS(X/V)$ is replaced by $\ACGSstab(X/V)$ and $\KimACGSstab(X/V)$ is replaced by $\Kimstab(X/V)$.  The relative stability condition~\ref{enum:stability} was imposed to remedy this:  we obtain a cartesian diagram by replacing $\KimACGSstab(X/V)$ with $\Kimstab(X/V)$ and $\ACGS(X/V)$ with $\ACGSstab(X/V)$.

\begin{nlem}[\namer{\ref{lem:maps}}{$\Upsilon$}{\ref{maps:cart}}]
Diagram~\eqref{eqn:7} is cartesian.
\end{nlem}
\begin{proof}
It is enough to produce a map
\begin{equation*}
\ACGSstab(X/V) \fp_{\ACGS(\sX/\sV)} \KimACGSstab(\sX/\sV) \rightarrow \Kimstab(X/V)
\end{equation*}
that is compatible with the maps to $\ACGSstab(X/V)$, $\KimACGSstab(\sX/\sV)$, and $\ACGS(\sX/\sV)$.  Lemma~\ref{lem:kimtoacgs}, combined with the projection $\KimACGS(X/V) \rightarrow \Kim(X/V)$ provides us with a map
\begin{equation*}
\ACGSstab(X/V) \fp_{\ACGS(\sX/\sV)} \KimACGSstab(\sX/\sV) \subset \KimACGS(X/V) \rightarrow \Kim(X/V).
\end{equation*}  We must verify that the map factors through $\Kimstab(X/V)$ and that it has the requisite compatibilities.

The factorization through $\Kimstab(\sX/\sV)$ follows from the following slightly stronger statement:  \emph{Consider a diagram~\eqref{eqn:8} and assume that the induced point of $\ACGS(X/V)$ is stable.  Then the induced point of $\Kim(X/V)$ is stable if and only if the induced point of $\KimACGS(X/V)$ is stable.}  

Let $\xi$ denote the point of $\Kim(X/V)$ induced from diagram~\eqref{eqn:8}, and let $G$ be its automorphism group.  Let $G''$ be the automorphism group of its image in $\ACGS(X/V)$, which is finite by hypothesis.  Since $\Upsilon$ is a functor, we have a homomorphism of groups $G \rightarrow G''$.  The kernel consists of all automorphisms of diagram~\eqref{eqn:8} that induce the identity on $\Upsilon(\xi)$.  This is the automorphism group of $C$ over $\oC \fp_{X_S} X^{\exp}_S$.  But $\oC \fp_{X_S} X^{\exp}_S = \oC \fp_{\sX_S} \sX^{\exp}$, so the kernel is precisely the automorphism group of $C$ over $\oC \fp_{\sX_S} \sX^{\exp}$.  Denoting this latter group by $G'$ we obtain an exact sequence
\begin{equation*}
1 \rightarrow G' \rightarrow G \rightarrow G''
\end{equation*}
As we have assumed $G''$ is finite, it follows that $G'$ is finite if and only if $G$ is.

This gives us the factorization.  The compatibility with the maps to $\ACGSstab(X/V)$ and $\ACGS(\sX/\sV)$ is immediate.  Compatibility with the map to $\KimACGSstab(\sX/\sV)$ amounts to the assertion that, \emph{if diagram~\eqref{eqn:8} is in $\ACGSstab(X/V) \fp_{\ACGS(\sX/\sV)} \KimACGSstab(\sX/\sV)$ then $\oC$ can be recovered as the stabilization of the map $C \rightarrow X_S$.}  But $\oC \rightarrow X_S$ is stable by the stability condition of $\ACGSstab(X/V)$, so the stabilization of $C \rightarrow X_S$ is the same as the stabilization of $C \rightarrow \oC$, which is $\oC$, by the stability condition of $\KimACGSstab(\sX/\sV)$.
\end{proof}

\begin{proposition} \label{prop:kimacgs-lfp}
$\KimACGS(\sX/\sV)$ is an algebraic stack locally of finite presentation.
\end{proposition}
\begin{proof}
We can prove this relative to the algebraic stack $\Log(\ACGS(\sX/\sV)) \fp_{\Log(\sV)} \Kim(\sX/\sV)$, which is locally of finite presentation by \cite[Theorem~1.1]{Olsson_log}.  It therefore suffices to show the following:  \emph{Suppose given a logarithmic commutative diagram
\begin{equation} \label{eqn:12} \vcenter{\xymatrix{
C \ar[r] \ar@{-->}[d]^\tau \ar@/_15pt/[dd]_\pi  & \sX^{\exp} \ar[dr] \ar[d]  \\
\oC \ar[r] \ar[d] & \sX \fp_{\sV} \sV^{\exp} \ar[r] \ar[d] & \sX \ar[d] \\
S \ar[r] & \sV^{\exp} \ar[r] & \sV .
}} \end{equation}
Then the dashed arrows $C \rightarrow \oC$ rendering the whole diagram an object of $\KimACGS(\sX/\sV)$ are parameterized by an algebraic space over $S$ that is locally of finite presentation.}

By \cite[Theorem~2.1.10 and section~2.4]{Chen}, the logarithmic maps $C \rightarrow \oC$ are parameterized by an algebraic space locally of finite presentation over $S$.  We can therefore assume that the dashed arrow in diagram~\eqref{eqn:12} is given.  The proposition therefore comes down to showing that the sheaf on $S$ parameterizing \emph{logarithmic} commutativity data for the diagram is representable by an algebraic space over $S$.  There is certainly such an algebraic space $Z$ parameterizing commutativity data over $C$, since all of the stacks appearing in the diagram are algebraic.

Now by Corollary~\ref{cor:log-stacks}, $\sX^{\exp}$, $\sX$, $\sV^{\exp}$, and $\sV$ all represent \'etale algebraic \emph{spaces} (not stacks) on logarithmic schemes.  That is, if $F$ is any of the stacks just named then $\uHom(C,F)$, the stack of logarithmic morphisms $C \rightarrow F$, is representable by an \emph{algebraic space} that is \emph{\'etale} over $C$.  Therefore if the diagram commutes, it commutes uniquely; moreover, $Z$ is the locus where some collections of sections of an \'etale morphism agree, hence is an open subscheme of $C$.  Since $\pi$ is proper, it follows that $\pi_\ast Z$ is open in $S$.  In particular, it is representable and locally of finite presentation.
\end{proof}

\begin{nlem}[\namer{\ref{lem:maps}}{$\Upsilon$}{\ref{maps:DM}}] \label{lem:kim-acgs-DM}
The map $\tUpsilon : \KimACGSstab(\sX/\sV) \to \ACGS(\sX/\sV)$ is of Deligne--Mumford type.
\end{nlem}
\begin{proof}
Suppose that $S$ is the spectrum of an algebraically closed field and suppose that $\xi$ is an $S$-point of $\KimACGSstab(\sX/\sV)$ corresponding to a diagram~\eqref{eqn:8}.  The stability condition~\ref{enum:stability} of Definition~\ref{def:kimacgs} is precisely the condition that the automorphism group of the underlying curve and map of $\xi$ (ignoring the logarithmic structures) that fixes its image in $\ACGS(\sX/\sV)$ be finite.  But by \hyperref[lem:kim-li-dm]{Lemma~\namerr{lem:maps}{$\Theta$}{maps:DM}}, if the underlying curve and map of $\xi$ have finite automorphism group, then so does $\xi$.
\Jonathan{I rewrote this a little---please make sure it is okay; maybe more details should be included}
%
%Let $S$ be the spectrum of an algebraically closed field.  It is enough to show that given an $S$ point of $\ACGS(\sX/\sV)$, the objects of $\KimACGSstab(\sX/\sV)\fp_{\ACGS(\sX/\sV)} S$ are all stable.  As in the proof of Proposition~\ref{prop:kimacgs-lfp}, this boils down to showing that for any $S$-scheme $T_S$, the group $G$ of automorphisms of an object $\gamma$ in $\KimACGSstab(\sX/\sV)(T_S)$ inducing the identity automorphism on its image $p(\gamma)$ in $\ACGS(\sX/\sV)(T_S)$ is finite.  Since the log.\ structure on the base is minimal, by \cite[Lemma~3.8.3]{Chen} and \cite[Section~6.3.1, Remark~2]{Kim}, $G$ is finite when the group $\uG$ of automorphisms of the underlying diagrams of schemes is finite.  This is given by the stability of $C \rightarrow \sX^{\exp}_S \fp_{\sX_S} \oC$.
\end{proof}

Unlike $\Psi$ and $\Theta$, for which the proofs of Lemma~\ref{lem:maps}~\ref{maps:etale} and~\ref{maps:birat} were automatic, $\Upsilon$ will require us to do some work to check the rest of Lemma~\ref{lem:maps}.  These
\Jonathan{changed ``the'' to ``these''}
remaining parts will follow easily from results proved in section~\ref{sec:kim} and Appendix~\ref{app:chains}.

\begin{nlem}[\namer{\ref{lem:maps}}{$\Upsilon$}{\ref{maps:etale}}] \label{prop:kimacgs-kim-etale}
The natural projection $\KimACGS(\sX/\sV)  \rightarrow \Kim(\sX/\sV)$ is \'etale and strict.
\end{nlem}
\begin{proof}
Strictness was part of the definition.  Since we already know both stacks are algebraic and locally of finite presentation, we only have to check that it is formally \'etale.

Suppose given an $S$-point $\xi$ of $\KimACGS(\sX/\sV)$, inducing an $S$-point $\eta$ of $\Kim(\sX/\sV)$, and suppose that $\eta'$ is an infinitesimal extension of $\eta$ to $S'$.  Let $\tau : C \rightarrow \oC$ be the projection coming from $\xi$
\Jonathan{added a few words}
and let $C'$ be the infinitesimal extension of $C$ corresponding to $\eta'$.  Write $\pi : C \rightarrow S$ and $\opi : \oC \rightarrow S$ for the two projections, so we have $\pi = \opi \circ \tau$.

Let $J$ be the ideal of $S$ in $S'$; thus $\pi^\ast J$ is the ideal of $C$ in $C'$ and we would like to construct an extension $\oC\vphantom{C}'$ with ideal $\opi^\ast J$.
\Jonathan{added ``we would like''}
By Lemma~\ref{lem:stabpf}, we have $R^1 \tau_\ast \pi^\ast J = \opi^\ast J \tensor R^1 \tau_\ast \tau^\ast \cO_{C} = 0$ since the fibers of $C$ over $C'$ are either points or rational curves.  This implies that $\tau_\ast \cO_{C'}$ is a square-zero extension of $\tau_\ast \cO_{C} = \cO_{\oC}$ by $\tau_\ast \pi^\ast J = \opi^\ast J$.  This is the structure sheaf of a scheme $\oC\vphantom{C}'$ over $S'$.  Moreover, by Theorem~\ref{thm:log-pf}, $M_{\oC\vphantom{C}'} = \tau_\ast M_{C'}$ is a log.\ structure on $\oC$ extending $M_{\oC}$ and we get a logarithmically commutative diagram,
\Jonathan{changed $\oC'$ to $\oC\vphantom{C}'$}
\begin{equation*} \xymatrix{
(\oC\vphantom{C}', M_{\oC\vphantom{C}'} ) \ar[r] \ar[d] & \sX \ar[d] \\
(S', M_{S'}) \ar[r] & \sV 
} \end{equation*}
where $M_{S'}$ is the logarithmic structure on $S'$ induced from the map $S' \rightarrow \Kim(\sX/\sV)$.  This gives the sought after lift to $\KimACGS(\sX/\sV)$ and proves that the projection is smooth.  

The universal property of push-forward ensures that there is a morphism between any two such lifts; as a morphism of extensions of algebras or of monoids with a fixed kernel is necessarily an isomorphism, the lift constructed above is unique up to a unique isomorphism.  The projection is therefore unramified as well.
\end{proof}

\begin{nlem}[\namer{\ref{lem:maps}}{$\Upsilon$}{\ref{maps:birat}}]
The map $\KimACGS(\sX/\sV) \rightarrow \ACGS(\sX/\sV)$ is generically an isomorphism.
\end{nlem}
\begin{proof}
It is sufficient to argue that the totally nondegenerate objects of $\KimACGS(\sX/\sV)$ form a dense open substack.  But these are the pullback, via the \'etale projection $\KimACGS(\sX/\sV) \rightarrow \Kim(\sX/\sV)$, of the totally nondegenerate objects of $\Kim(\sX/\sV)$, which are dense by \hyperref[cor:kim-dense]{\namer{Lemma~\ref*{lem:spaces}}{$\Kim$}{\ref*{spaces:dense}}}.
\end{proof}

\appendix
\numberwithin{theorem}{section}

\section{Logarithmic structures}

\begin{proposition} \label{prop:loc-free}
The stack of locally free logarithmic structures is smooth and connected.
\end{proposition}
\begin{proof}
Suppose $S' = \Spec A'$ is an infinitesimal extension of $S = \Spec A$.  Let $M$ be a locally free logarithmic structure on $S$.  After \'etale localization, we can assume that $M$ has a chart $\bN^r \rightarrow A$.  Since $A' \rightarrow A$ is surjective, this lifts to $\bN^r \rightarrow A'$.

To see the connectedness, consider a point of the stack of locally free logarithmic structures.  It corresponds to a map $\bN^r \rightarrow k$ for some field $k$.  This extends to a map $\bN^r \rightarrow k[t_1, \ldots, t_r]$, which is generically trivial.  Therefore the trivial logarithmic structures are dense in the stack of all locally free logarithmic structures.
\end{proof}

\begin{proposition}
The stack $\sA$, with its natural logarithmic structure, represents the functor $(S, M_S) \mapsto \Gamma(S, \oM_S)$ on the category of logarithmic schemes.
\end{proposition}
\begin{proof}
With its usual logarithmic structure $\bA^1$ represents the functor $(S, M_S) \mapsto \Gamma(S, M_S)$.  This carries an action of $\Gm$, by which $\oM_S$ is the quotient.
\end{proof}

\begin{remark}
Even though the underlying ``space'' of $\sA$ is a stack, it represents a sheaf on the category of logarithmic schemes!  See \cite[Proposition~5.17]{Olsson_log} for a more general statement.
\Jonathan{added second sentence}
\end{remark}

\begin{corollary} \label{cor:sA-etale}
For all $n$, the stacks $\sA^n$ are \'etale over $\Log$ when equipped with their natural logarithmic structures.
\end{corollary}
\begin{proof}
Form the fiber product $\sA^n \fp_{S} \Log$ for any logarithmic scheme $(S, M_S)$.  If $f : T \rightarrow S$ is a morphism of schemes, viewed as a strict morphism of logarithmic schemes, then $\Hom_S(T, \sA^n \fp_S \Log) = \Gamma(T, f^\ast \oM_S^n)$.  As $\oM_S$ is an \'etale sheaf, this functor is represented by an algebraic space \'etale over $S$.
\end{proof}

\begin{corollary} \label{cor:sA-open}
$\sA$ is open in $\Log$.
\end{corollary}
\begin{proof}
We have just seen that $\sA$ is \'etale over $\Log$, so we only need to check $\sA \rightarrow \Log$ is fully faithful.  On the open point, this is entirely obvious.  The automorphism group of the closed point is $\Gm$, which is the same as the automorphism group of the logarithmic structure $\bN \times \cO^\ast \rightarrow \cO$ sending $(n, \lambda)$ to $\lambda \, 0^n$.
\end{proof}

\begin{corollary} \label{cor:log-stacks}
The stacks $\sX$, $\sV$, $\sX^{\exp}$, $\sV^{\exp}$, and $\sA^n$ for all $n$ are \'etale over $\Log$.  All logarithmic morphisms between these stacks are logarithmically \'etale.
\end{corollary}
\begin{proof}
All of the stacks in question are \'etale-locally isomorphic over $\Log$ to products of copies of $\sA$.  Therefore they are all \'etale over $\Log$.  Hence each of these stacks is logarithmically \'etale over a point.  Maps between logarithmically \'etale stacks are necessarily logarithmically \'etale, as the following lemma demonstrates.
\Jonathan{changed proof}
\end{proof}

\begin{lemma}
Let $F$ and $G$ be logarithmically \'etale algebraic stacks.  Then any logarithmic morphism $F \rightarrow G$ is logarithmically \'etale.
\Jonathan{added lemma}
\end{lemma}
\begin{proof}
Suppose we have a logarithmic lifting problem
\begin{equation*} \xymatrix{
S \ar[r] \ar[d] & F \ar[d] \\
S' \ar@{-->}[ur] \ar[r] & G
} \end{equation*}
where $S'$ is a strict infinitesimal extension of $S$ then there is a unique extension of the map $S \rightarrow F$ to a map $S' \rightarrow F$ because $F$ is logarithmically \'etale over a point.  This gives the commutativity of the upper triangle.  The commutativity of the lower triangle follows for the same reason:  there is a unique extension of the map $S \rightarrow G$ to a map $S' \rightarrow G$ because $G$ is logarithmically \'etale over a point.
\end{proof}

%This means that if $S$ is a log.\ scheme and $F$ is any one of the stacks above, then $F \fp_{\Log} S$ is representable on the category of strict log.\ schemes over $S$ by an algebraic space that is \'etale over $S$:  there is an algebraic space $F_S$ over $S$ such that if $T \rightarrow S$ is a strict morphism of log.\ schemes, to give a map $T \rightarrow F$ is the same as to give a section of $F_S$ over $T$.

%Let $F$ and $G$ be two of the stacks considered above and $F \rightarrow G$ a logarithmic morphism.  Consider an infinitesimal lifting problem
%\begin{equation*} \xymatrix{
%S \ar[r] \ar[d] & F \ar[d] \\
%S' \ar@{-->}[ur] \ar[r] & G
%} \end{equation*}
%in which $S \subset S'$ is a strict infinitesimal logarithmic extension.  Let $F_{S'}$ and $G_{S'}$ be as in the last paragraph.  As these are both \'etale over $S'$, a section of $F_{S'}$ or $G_{S'}$ over $S$ extends uniquely to one over $S'$.  It follows that the lifting problem considered above has a unique solution and therefore that the map $F \rightarrow G$ is \'etale.\Jonathan{I tried rewriting this to make it clearer; probably needs to be rewritten again}

\section{Chains of rational curves}
\label{app:chains}

Let $C$ be a pre-stable curve and $\oC$ a partial stabilization of $C$ over a base $S$.  We have a projection $\tau : C \rightarrow \oC$ that contracts chains of rational curves.  

\begin{lemma} \label{lem:stabpf}
The natural map $\cO_{\oC} \rightarrow \RR \tau_\ast \cO_{C}$ is a quasi-isomorphism.
\end{lemma}
\begin{proof}
By the theorem on formal functions \cite[Th\'eor\`eme~(4.1.5)]{ega3-1} and a standard reduction using local finite presentation, it will be sufficient to prove the lemma when $S$ is the spectrum of an artinian local ring.  It is clearly true when $S$ is a point:  the fibers of $C$ over $\oC$ are either points or chains of rational curves and in either case $\RR^i \tau_\ast \cO_C = 0$ for $i > 0$.

We proceed by induction on the length of $S$.  Suppose that $S'$ is a small extension of $S$ and that $\oC\vphantom{C}'$, $C'$, etc.\ are extensions of the relevant data to $S'$.  Assuming that $\RR \tau_\ast \cO_C = \cO_{\oC}$, we show that $\RR \tau_\ast \cO_{C'} = \cO_{\oC\vphantom{C}'}$.

In this case the ideal of $C$ in $C'$ is isomorphic to $\cO_{C_0}$ where $C_0$ is the central fiber of $C$ over $S$.  By the inductive hypothesis, we have $R^i \tau_\ast \cO_{C_0} = R^i \tau_\ast \cO_C = 0$ for $i > 0$.  This, combined with the long exact sequence for $R \tau_\ast$ implies that $R^i \tau_\ast \cO_{C'} = 0$ for $i > 0$ as well.
\end{proof}

\begin{corollary}
The natural map $\cO_{\oC}^\ast \rightarrow R \tau_\ast \cO_C^\ast = \cO_{\oC}^\ast$ is an isomorphism.
\end{corollary}
\begin{proof}
  We certainly have $\tau_\ast \cO_C^\ast = \cO_{\oC}^\ast$ from Lemma~\ref{lem:stabpf}.  We only have to check that $R^1 \tau_\ast \cO_C^\ast = 0$, or phrased another way, that every $\Gm$-torsor on $C$ is pulled back from a $\Gm$-torsor on $\oC$.  This is clear when $S$ is a point; it follows in general because first-order deformations of line bundles on $C$ and $\oC$ are both classified by $H^1(C, \cO_C) = H^1(\oC, \cO_{\oC})$. 
\end{proof} 

Assume that $S$ and $C$ are equipped with logarithmic structures $M_S$ and $M_C$, respectively, making $C$ log.\ smooth over $S$.

\begin{lemma}
$\tau_\ast M_C$ is a logarithmic structure on $\oC$.
\end{lemma}
\begin{proof}
We have a map $\tau_\ast M_C \rightarrow \tau_\ast \cO_C = \cO_{\oC}$ by pushforward.  We have to check it gives a bijection on units.  We have $\cO_{\oC}^\ast = \tau_\ast \cO_{C}^\ast \subset \tau_\ast M_C$, again by pushforward.  Furthermore,
\begin{equation*}
\tau_\ast M_C / \tau_\ast \cO_C^\ast = \tau_\ast (M_C / \cO_C^\ast) .
\end{equation*}
Since $M_C$ is a log.\ structure, $M_C / \cO_C^\ast$ has no units other than the identity.  Therefore $\tau_\ast M_C / \cO_{\oC}^\ast$ contains no units other than the identity, and $\cO_{\oC}^\ast$ is therefore the group of units in $\tau_\ast M_C$.
\end{proof}

For the rest of this section, we write $M_{\oC} = \tau_\ast M_C$.

\begin{lemma} \label{lem:base-change}
Let $f : S' \rightarrow S$ be a strict morphism  of log.\ schemes, and let $C'$ and $\oC\vphantom{C}'$ be the log.\ schemes obtained by base change, and $\tau' : C' \rightarrow \oC\vphantom{C}'$ the induced projection.  Then $f^\ast M_{\oC} = M_{\oC\vphantom{C}'}$.
\end{lemma}
\begin{proof}
We have 
\begin{equation*}
f^\ast M_C = f^{-1} M_C \times \cO_{C'}^\ast / f^{-1} \cO_C^\ast 
\end{equation*}
since $f^{-1} \cO_C^\ast$ is the group of units in $f^{-1} M_C$.  Because $R^1 \tau'_\ast f^{-1} \cO_C^\ast = 0$, we have
\begin{equation*}
\tau_\ast f^\ast M_C = f^{-1} \tau_\ast M_C \times \tau_\ast \cO_{C'}^\ast / f^{-1} \tau_\ast \cO_C^\ast
\end{equation*}
and, noting that $\tau_\ast \cO_{C'}^\ast = \cO_{\oC}^\ast$ and $\tau_\ast \cO_C^\ast = \cO_{\oC}^\ast$, this is precisely $M_{\oC\vphantom{C}'}$.
\end{proof}

\begin{lemma} \label{lem:log-smooth}
$(\oC, M_{\oC})$ is log.\ smooth over $S$.
\end{lemma}
\begin{proof}
Since $\oC$ is flat over $S$, it is sufficient to prove this on the geometric fibers.  We can therefore assume that $S$ is the spectrum of a separably closed field and that $M_S$ is the log.\ structure associated to a morphism of monoids $P \rightarrow \Gamma(S, \cO_S)$.  Let $S' = \Spec \cO_S[P]$, with its natural log.\ structure, so that we have a strict map $S \rightarrow S'$.  After replacing $S'$ with an \'etale neighborhood of $S$, we can assume that $C$ extends to a family of pre-stable curves $C'$ over $S'$.  

Let $\oC\vphantom{C}'$ be the family of curves obtained by contracting the components of $C'$ corresponding to the components of $C$ that are contracted in $\oC$.  By Lemma~\ref{lem:base-change}, it will be sufficient to prove the lemma with $S$ replaced by $S'$, $C$ by $C'$, etc.  But $S'$, $C'$, and $\oC\vphantom{C}'$ are all toroidal with the toroidal log.\ structures, so the lemma is proved.
\end{proof}

\begin{theorem} \label{thm:log-pf}
Let $(C,M_C) \rightarrow (S,M_S)$ be a family of log.\ smooth curves and let $C \rightarrow \oC$ be a morphism over $S$ whose stabilization is an isomorphism ($\tau$ contracts semistable components of $C$).  Then 
\begin{enumerate}[label=(\roman{*})]
\item $\tau$ extends uniquely to a log.\ morphism when $\oC$ is given the log.\ structure making it log.\ smooth over $(S,M_S)$, and 
\item if $f : (C,M_C) \rightarrow (X,M_X)$ is any log.\ morphism whose underlying map $C \rightarrow X$ factors through $\oC$ then the log.\ map $f$ factors through $(\oC, M_{\oC})$.
\end{enumerate}
\end{theorem}
\begin{proof}
We have $M_{\oC} = \tau_\ast M_C$, which proves the uniqueness.  The factorization comes from the adjunction $(\tau^\ast, \tau_\ast)$ applied to the morphism of monoids $\tau^\ast \of^\ast M_X = f^\ast M_X \rightarrow M_C$.
\end{proof}

\section{Tangent bundles of stacks}
\label{app:tangent}

Let $S$ be a scheme.  Write $S[\epsilon]$ for the trivial square-zero extension of $S$ with ideal $\cO_S$.  Recall that the tangent bundle of a scheme $X$ is by definition the scheme $T_X$ such that
\begin{equation*}
\Hom(S, T_X) = \Hom(S[\epsilon], X) .
\end{equation*}
We may employ the same definition for the tangent bundle of an algebraic stack.  The tangent bundle of $X$ is therefore a stack over $X$.  Because an algebraic stack is homogeneous \cite[Proposition~2.1]{obs}, the tangent bundle stack comes equipped with the structure of a commutative $2$-group (see, e.g., \cite[Proposition~2.2]{obs}).

\begin{proposition} \label{prop:TBGm}
The tangent bundle of $\BGm$ is isomorphic to the vector bundle stack $\BGa \times \BGm$ over $\BGm$.
\end{proposition}
\begin{proof}
Let $S$ be a scheme, $S'$ the trivial square-zero extension with ideal $\epsilon \cO_S$, and $L$ the line bundle corresponding to an $S$-point of $\BGm$.  If $L'$ is an extension of $L$ to a line bundle on $S'$ then $\uIsom_L(L + \epsilon L, L')$ is a torsor under $\uHom(L,L) = S \times \Ga$.

Now suppose $P$ is a $\Ga$-torsor.  We construct an extension $L'$ of $L$ to $S'$ by contracting the trivial extension $L + \epsilon L$ with $P$ via their $\Ga$-actions.  These constructions are easily seen to be mutually inverse.

As these constructions are functorial in both $S$ and $L$, they give an isomorphism between the tangent bundle of $\BGm$ and $\BGa \times \BGm$. 
\end{proof}

\begin{proposition} \label{prop:TsA}
There is an equivalence between $\sA$ and the stack parameterizing pairs $(L, s)$ where $L$ is a line bundle and $s$ is a section of $L$.  Under this identification, the tangent bundle of $\sA$ is identified with the vector bundle stack $[\,L\, /\, \Ga \,]$ over $\sA$ associated to the complex $[ \cO \xrightarrow{s} L ]$ in degrees $[-1,0]$.
\end{proposition}
\begin{proof}
A section of the tangent bundle of $\sA$ over an $S$-point $(L,s)$ of $\sA$ corresponds to an extension of the line bundle and section $(L, s)$ associated to the map $S \rightarrow \sA$ to a line bundle and section $(L',s')$ on $S'$, the trivial square-zero extension of $S$ by $\cO_S$.  We know from Proposition~\ref{prop:TBGm} that the extension $L'$ of $L$ is classified by a $\BGa$-torsor.  The sections of $L'$ form a torsor under $L$, so there is an $\sA$-morphism
\Jonathan{small wording change}
$\BGa \times \sA \rightarrow BL$ whose fiber is the tangent bundle of $\sA$.  In order to identify the tangent bundle precisely, we will show that the map $\BGa \times \sA \rightarrow BL$ is the one induced from the section $s : \Ga \rightarrow L$.

The $L$-torsor of sections of $L'$ may be realized as the fiber over $s : S \rightarrow L$ of the projection $L' \rightarrow L$.  If $P$ is the $\Ga$-torsor associated to $L$ (as in Proposition~\ref{prop:TBGm}) then we can identify the $L$-torsor associated to $s$ as
\begin{equation*}
L' \fp_{L} (S,s) = \bigl( (L + \epsilon L) \fp^{\Ga} P \bigr) \fp_{L} (S,s) = (s + \epsilon L) \fp^{\Ga} P
\end{equation*}
where we have written $s + \epsilon L$ for the fiber of $L + \epsilon L$ over $s$.  Of course, $s + \epsilon L$ is isomorphic to $\epsilon L \simeq L$ if the action of $\Ga$ is ignored; however, the action is given by 
\begin{equation*}
t \: . \: (s + \epsilon x) = (1 + \epsilon t)(s + \epsilon x) = s + \epsilon (x + st).
\end{equation*}
Thus $L' \fp_L (S, s) \simeq P \tensor_{\Ga} (L, s)$ is the $L$-torsor induced from $P$ via the homomorphism $s : \Ga \rightarrow L$.

The tangent bundle of $\sA$ is therefore the fiber of the map $\BGa \rightarrow BL$ induced from $s$, which is $[\,L\, /\, \Ga\,]$ with $\Ga$ acting via $s$.
\end{proof}

%\vfill\eject
%\newgeometry{margin=0.75in}

\section{Notation Index}\label{sec:notationindex} 

\makenotn

%\restoregeometry

\bibliographystyle{amsalpha}             % (uses file "plain.bst")
\bibliography{logrel}       % expects file "myrefs.bib"
\end{document}